\documentclass[11pt]{amsart}
\usepackage{fullpage}   

\usepackage{amsmath}
\usepackage{amsfonts}
\usepackage{amssymb}
\usepackage{amsthm}
\usepackage{subfigure}
\usepackage{graphicx}
\usepackage{color}
\usepackage{pinlabel}
\usepackage{hyperref}
\usepackage{comment}

\newtheorem{theorem}{Theorem}
\newtheorem{proposition}{Proposition}
\newtheorem{lemma}{Lemma}

\newtheorem{corollary}{Corollary}

\theoremstyle{definition}
\newtheorem{definition}{Definition}

\theoremstyle{remark}
\newtheorem{remark}{Remark}

\newtheorem{open}{Open problem}

\begin{document}

\newcommand{\rp}{\mathbb{R}\mathrm{P}^2 }
\newcommand{\G}{\mathbb{G}}
\newcommand{\RP}{$\mathbb{R}\mathrm{P}^2$ }
\newcommand{\RPt}{$\mathbb{R}\mathrm{P}^2$}
\newcommand{\bs}{\backslash}
\newcommand{\RPthree}{$\mathbb{R}\mathrm{P}^3$ }
\newcommand{\RPthreet}{$\mathbb{R}\mathrm{P}^3$}
\newcommand{\A}{\mathbb{A}}


\title[Graphs of links in projective space]{On the ribbon graphs of links in real projective space}

\author{Iain Moffatt}
\author{Johanna Str\"{o}mberg}
\address{Department of Mathematics, Royal Holloway University of London, Egham, Surrey, TW20 0EX, United Kingdom}
\email{iain.moffatt@rhul.ac.uk}
\email{anna.stromberg.2011@live.rhul.ac.uk}


\date{\today}

\begin{abstract}
Every link diagram can be represented as a signed ribbon graph. However, different link diagrams can be represented by the same ribbon graphs. We determine how checkerboard colourable diagrams of links in real projective space, and virtual link diagrams, that are represented by the same ribbon graphs are related to each other. We also find moves that relate the diagrams of links in real projective space that give rise to (all-A) ribbon graphs with exactly one vertex.
\end{abstract}

\maketitle


\section{Introduction and overview}
It is well-known that a classical link diagram can be represented by a unique signed plane graph, called its Tait graph (see, for example, the surveys \cite{MR1633290,MR3086663,MR1245272}). This construction provides a seminal connection between the areas of graph theory and knot theory, and has found impressive applications such as in  proofs of the Tait conjectures \cite{MR895570,MR899051}. Tait graphs can also be constructed  for checkerboard colourable link diagrams on other surfaces, in which case the resulting graph is embedded on the surface.  However, as this construction requires checkerboard colourability, Tait graphs cannot be constructed  for  arbitrary link diagrams on a surface, or arbitrary virtual link diagrams. Recently, Dasbach,  Futer, Kalfagianni, Lin, and Stoltzfus, in \cite{MR2389605}, extended the idea of a Tait graph by associating a set of signed ribbon graphs  to a link  diagram (see also Turaev \cite{MR925987}). Chmutov and Voltz extended this construction, giving a way to describe an arbitrary virtual link diagram as a signed ribbon graph in  \cite{MR2460170}. These constructions extend to graphs in other surfaces. The ribbon graphs of link diagrams have found numerous applications, and we refer the reader to the surveys \cite{Champanerkar:2014yq,MR3086663} for details. 

Every signed plane graph represents a unique classical link diagram. In contrast, a single signed ribbon graph can represent several different link diagrams or virtual link diagrams. This observation leads to the fundamental problem of determining  how link diagrams that are presented by the same signed ribbon graphs are related to each other. It is this problem that interests us here.  It was solved for classical link diagrams in \cite{MR2928906}. Here we solve it for checkerboard colourable diagrams of links in  \RPthree (in Theorem~\ref{mainthm}), and for virtual link diagrams (in Theorem~\ref{t.vmain}). 

We also examine the one-vertex ribbon graphs of diagrams of links  in \RPthreet. Every classical link diagram  can be represented as a ribbon graph with exactly one vertex. In  \cite{MR3148509}, Abernathy et al  gave a set of moves that provide a way to move between all of the diagrams of a classical link that have one-vertex all-A ribbon graphs. We extend their work to the setting of links in \RPthreet.

This paper is structured as follows. In Section~\ref{nt} we give an overview of diagrams of links in $\mathbb{R}\mathrm{P}^3$ and of ribbon graphs.  In Section~\ref{s3} we describe how diagrams of links in $\mathbb{R}\mathrm{P}^3$  can be represented by   ribbon graphs and we determine how checkerboard colourable diagrams that give rise to the same ribbon graphs are related to one another. In section~\ref{1v} we study the ribbon graphs of  diagrams of links in $\mathbb{R}\mathrm{P}^3$ that have exactly one vertex. Finally, in Section~\ref{s.virt}  we describe how virtual link diagrams that give rise to the same ribbon graphs are related to one another

This work arose from J.S.'s undergraduate thesis at Royal Holloway, University of London which was supervised by  I.M..

\section{Notation and terminology}\label{nt}

\subsection{Links in $\mathbb{R}\mathrm{P}^3$ and their diagrams}

In this section we provide a brief overview of links in $\mathbb{R}\mathrm{P}^3$ and their diagrams. Further results and details can be found in \cite{MR1296890,MR1073213,MR2414448,MR1999636,MR895570,MR1414898}.

A \emph{diagram} of a link in $\mathbb{R}\mathrm{P}^3$ is a disc $D^2$ in the plane together with a collection of immersed arcs (where an arc is a compact connected 1-manifold possibly with boundary). The end points of arcs with boundary lie on the boundary of the disc $\partial D^2$, are divided into antipodal pairs, and these are the only points of the arcs that intersect $\partial D^2$. We further assume that the arcs are generically immersed, in that they have finitely many multiple points and each multiple point is a  double point in which the arcs meet transversally.  Finally, each double point is assigned an over/under-crossing structure, and is called a \emph{crossing}. Figure~\ref{f1a} shows a diagram of a link in $\mathbb{R}\mathrm{P}^3$.  Here, $D$ will always refer to a diagram of a link in $\mathbb{R}\mathrm{P}^3$.

\begin{figure}
\centering
\subfigure[A diagram $D$ of a link in $\mathbb{R}\mathrm{P}^3$.]{
\includegraphics[scale=.4]{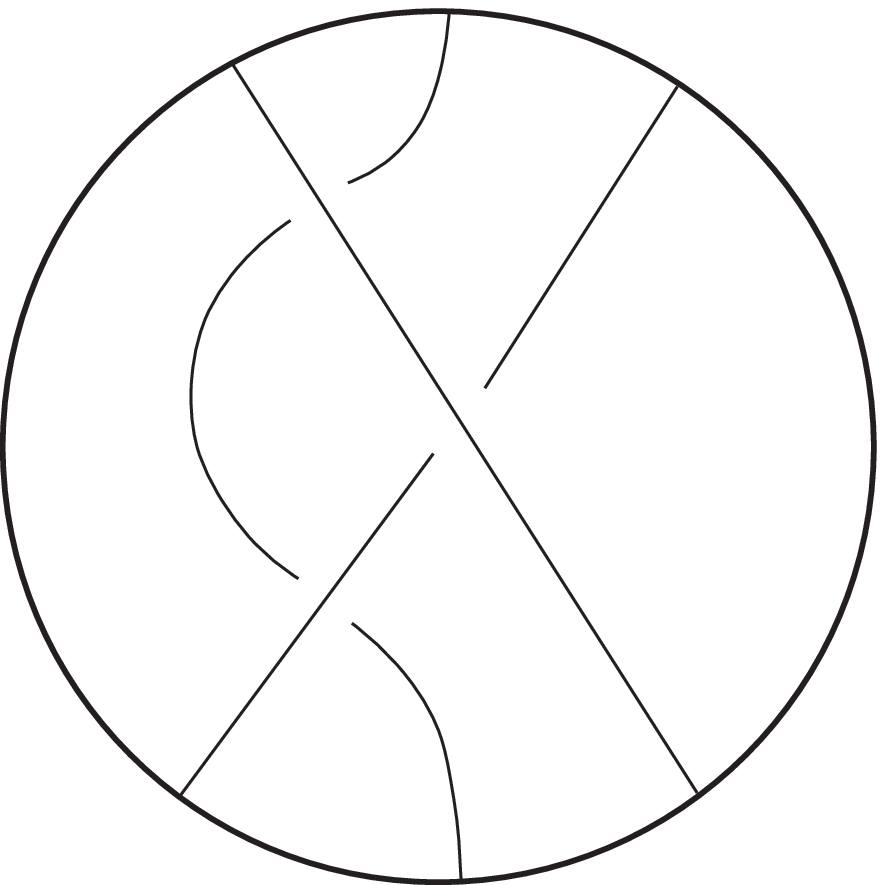}
\label{f1a}
}
\hspace{10mm}
\subfigure[A  state $\sigma$ of $D$ . ]{
\labellist
 \small\hair 2pt
\pinlabel {$a +$}  at 118 198
\pinlabel {$b -$}  at 145 145
\pinlabel {$c -$}  at 100 80
\endlabellist
\includegraphics[scale=.4]{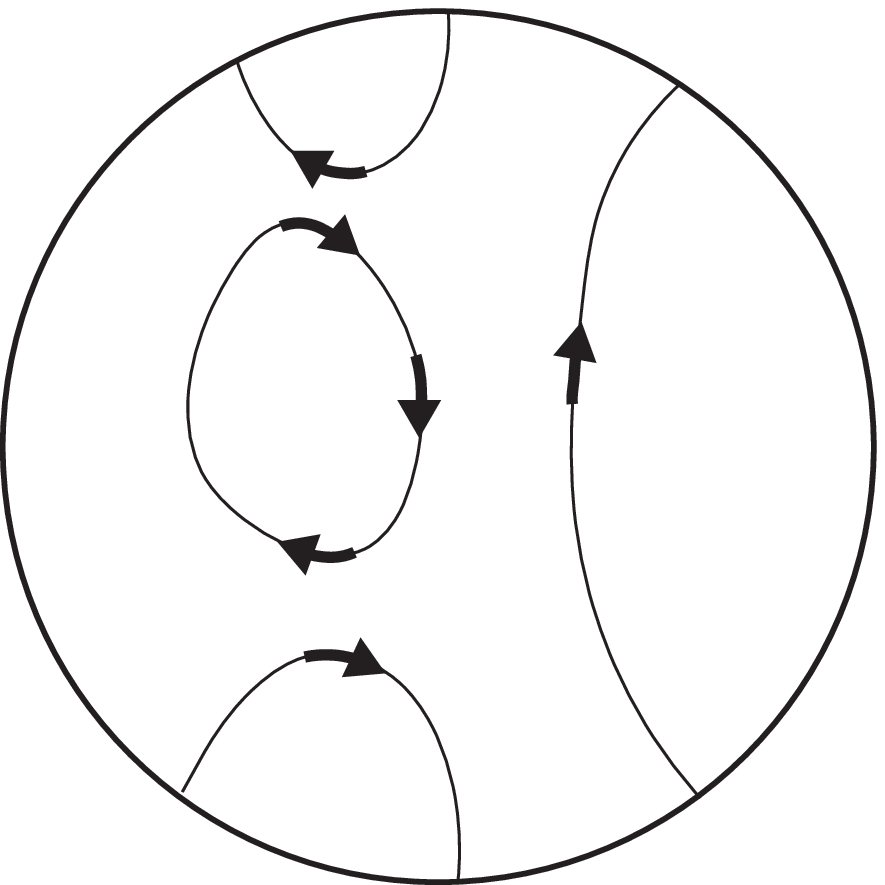}
\label{f1b}
}
\hspace{10mm}
\subfigure[Redrawing the arrow presentation for $G_{(D,\sigma)}$.]{
\labellist
 \small\hair 2pt
\pinlabel {$a +$}  at 59 165
\pinlabel {$a +$}  at 250 165
\pinlabel {$b -$}  at 140 90
\pinlabel {$b -$} at 330 90
\pinlabel {$c -$} at 55 20
\pinlabel {$c -$} at 250 20
\endlabellist
\includegraphics[scale=.4]{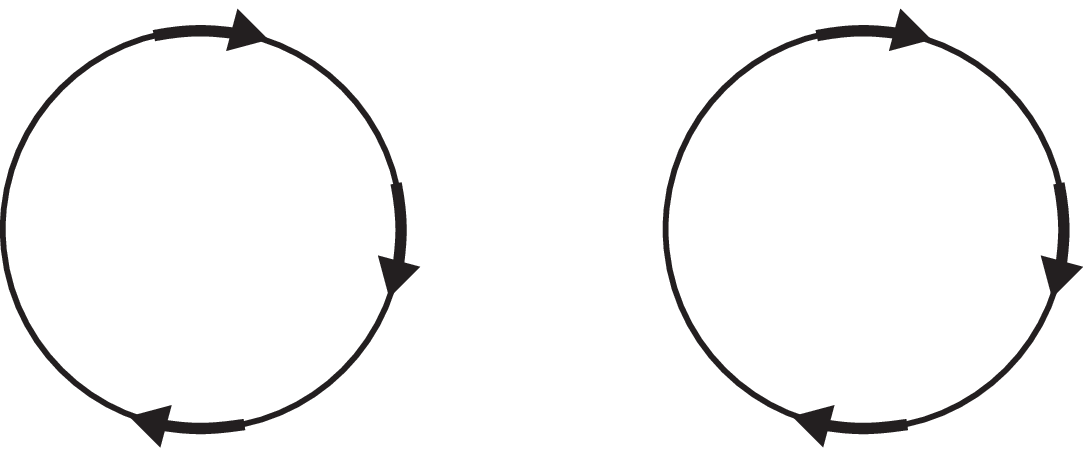}
\label{f1c}
}
\hspace{10mm}
\subfigure[ $G_{(D,\sigma)}$ as a ribbon graph. ]{
\labellist
 \small\hair 2pt
\pinlabel {$a +$}  at 123 195
\pinlabel {$b -$}  at 250 211
\pinlabel {$c -$}  at 150 15
\endlabellist
\includegraphics[scale=.35]{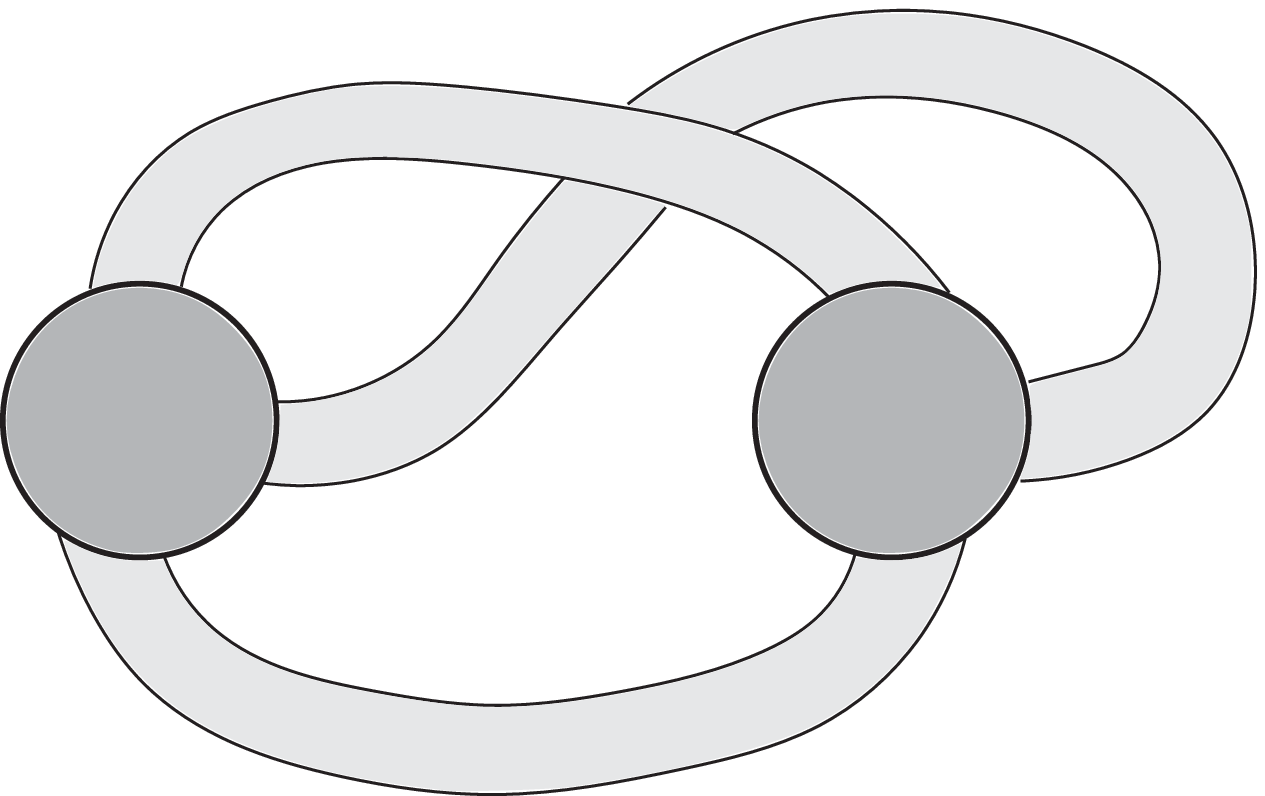}
\label{f1d}
}
\subfigure[Embedding $G_{(D,\sigma)}$ in a surface.]{
\includegraphics[scale=.4]{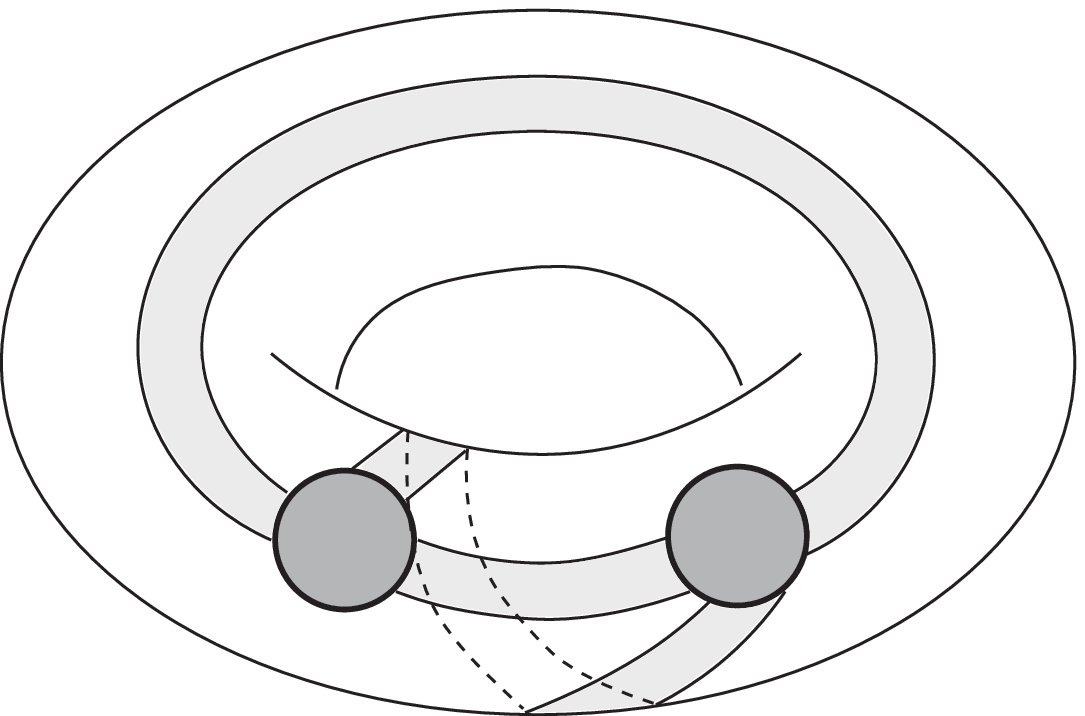}
\label{f1e}
}
\hspace{10mm}
\subfigure[ $G_{(D,\sigma)}$ as a cellularly embedded graph. ]{
\includegraphics[scale=.4]{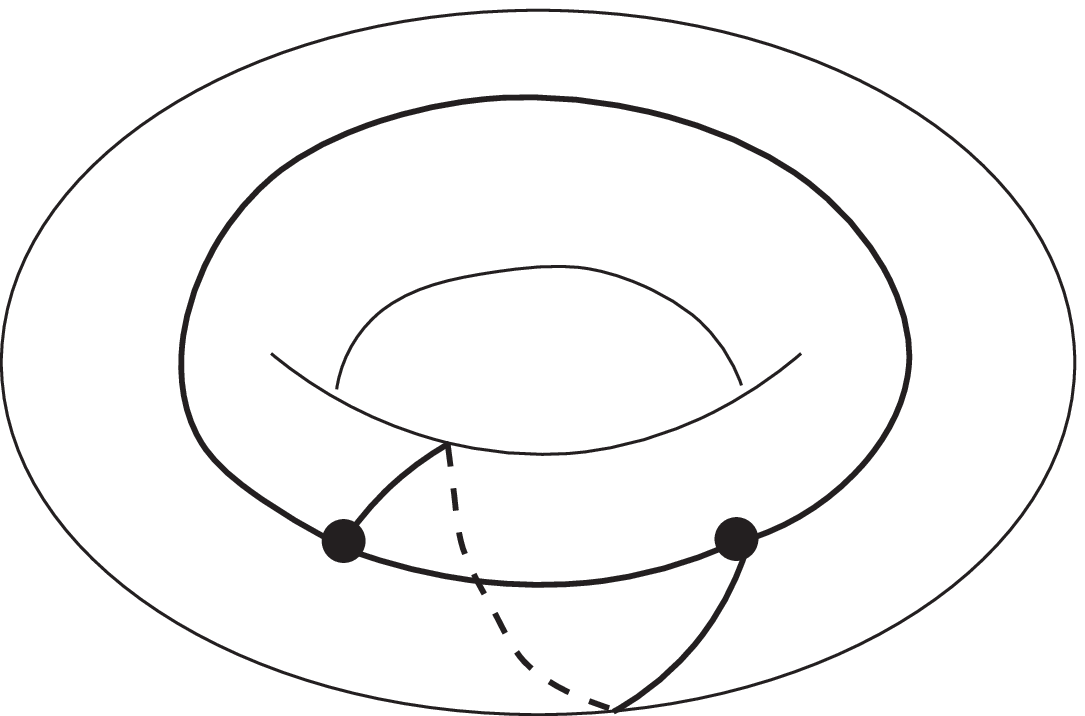}
\label{f1f}
}
\caption{A diagram $D$ of a link in $\mathbb{R}\mathrm{P}^3$ and one of its ribbon graphs.}
\label{f1}
\end{figure}

A \emph{net} is the real projective plane  $\mathbb{R}\mathrm{P}^2$ together with a distinguished projective line, called the \emph{line at infinity}, and a collection of generically immersed closed curve where each double point is assigned an over/under-crossing structure. Let $D$ be a diagram of a link in $\mathbb{R}\mathrm{P}^3$, then the \emph{net of $D$}, denoted $\mathcal{N}_D$, is obtained from $D$  by identifying the antipodal points of $\partial D^2$. The image of $\partial D^2 $ in the net gives the line at infinity.

A \emph{component} of $D$ is a collection of its arcs that give rise to a single closed curve in its net $\mathcal{N}_D$. A component is \emph{null-homologous} if the corresponding curve in $\mathcal{N}_D$ is trivial in $H_1(\mathbb{R}\mathrm{P}^2)=\mathbb{Z}_2$ and is \emph{1-homologous} otherwise.  We will say that a diagram is \emph{null-homologous}  if each of its components is. The \emph{faces} of $D$ (respectively, $\mathcal{N}_D$) are the components of $D\backslash  \alpha$  (respectively, $\mathcal{N}_D\backslash  \alpha$) where $\alpha$ is the set of immersed curves. A \emph{region} of $D$ is a collection of its faces that correspond to a single face in its net $\mathcal{N}_D$.
A diagram $D$ is \emph{checkerboard colourable} if there is an assignment of the colours black and white to its regions such that no two adjacent regions (those meeting a common arc) are assigned the same colour. A diagram may or may not be checkerboard colourable. For example, the diagram  in Figure~\ref{f1a} is not,  but that in Figure~\ref{f3d} is.

The \emph{Reidemeister moves} for diagrams of links in $\mathbb{R}\mathrm{P}^3$ consist of isotopy of the disc that preserves the antipodal pairing (which we call the \emph{R0-move}), together with the five moves  in Figure~\ref{f2} that change the diagram locally as shown (the diagrams are identical outside of the given region). In the figure, the bold lines represent the boundary of the disc. Two diagrams are \emph{equivalent} if they are related by a sequence of Reidemeister moves.

\begin{figure}[ht]
\centering
\subfigure[The classical moves.]{
\labellist
 \small\hair 2pt
\pinlabel {RIII}  at 108 50
\pinlabel {RII}  at 108 158
\pinlabel {RI}  at 108 264
\endlabellist
\includegraphics[scale=.6]{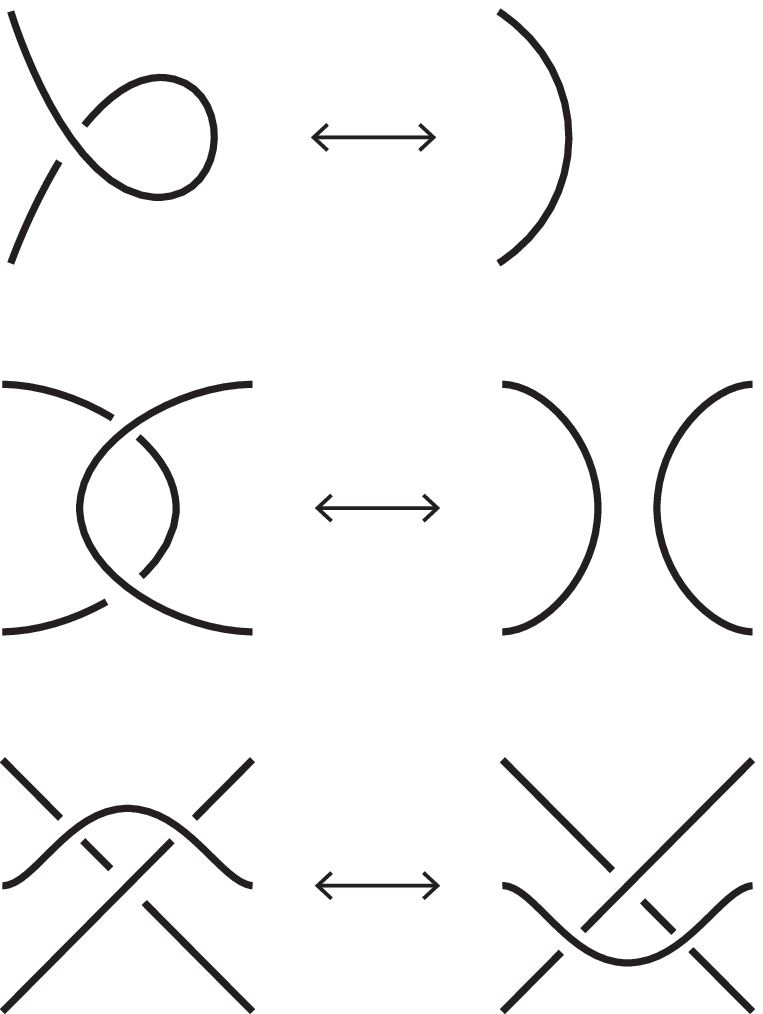}
\label{f2a}
}
\hspace{17mm}
\subfigure[The boundary moves.]{
\labellist
 \small\hair 2pt
\pinlabel {RV}  at 258 120
\pinlabel {RIV}  at 258 380
\endlabellist
\includegraphics[scale=.35]{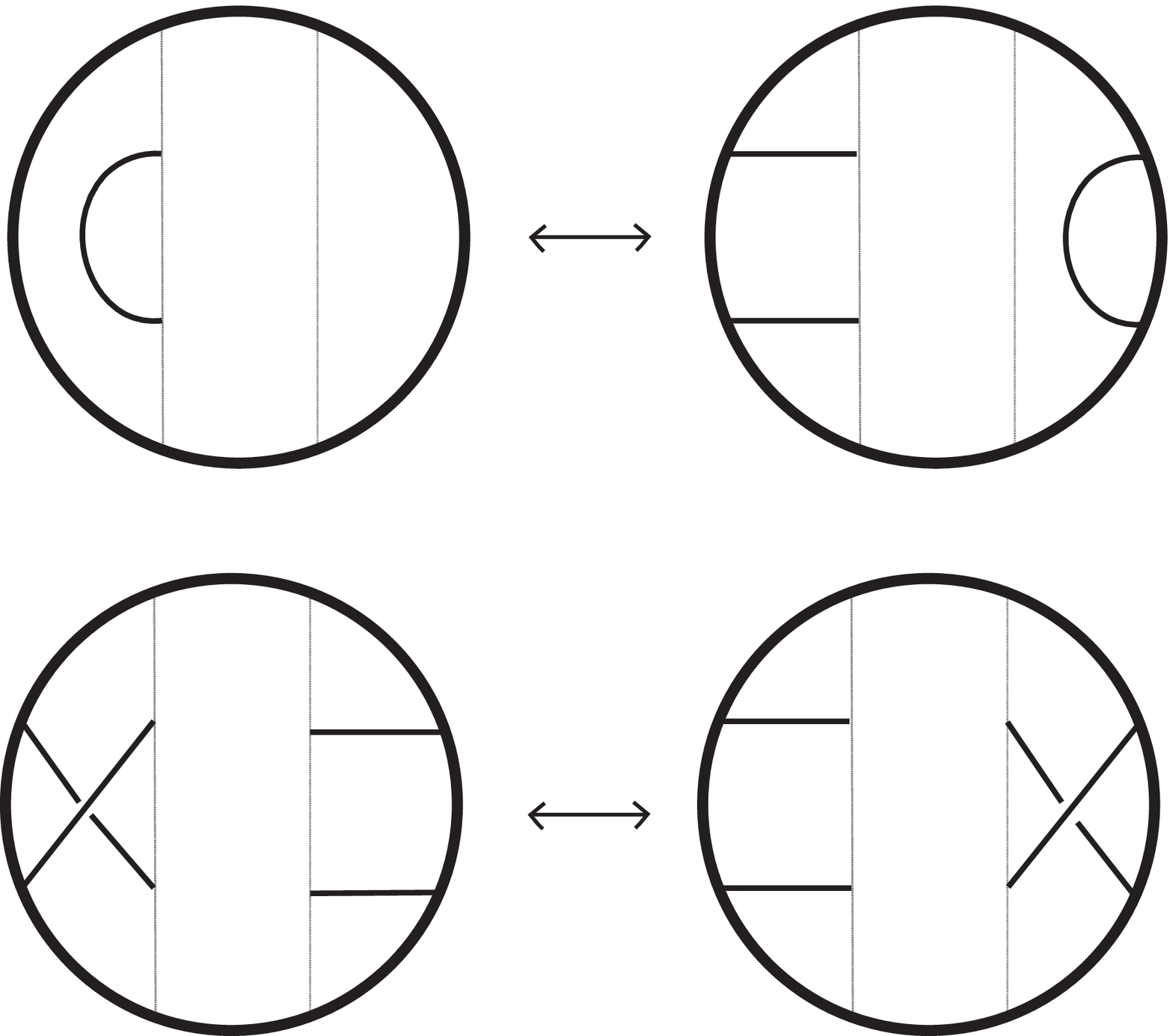}
\label{f2b}
}

\caption{The Reidemeister moves for diagrams of links in \RPthreet.}
\label{f2}
\end{figure}

For brevity we work a little informally in this paragraph, referring the reader to \cite{MR1073213} for details. 
Links in  $\mathbb{R}\mathrm{P}^3$  give rise to diagrams by representing $\mathbb{R}\mathrm{P}^3$ as a ball $D^3$ with antipodal points of its boundary   identified, lifting the link from $\mathbb{R}\mathrm{P}^3$ to  $D^3$ and projecting to the equatorial disc $D^2$. Conversely, given a diagram, regarding  $D^2$ as the equatorial disc of such a representation of  $\mathbb{R}\mathrm{P}^3$, and ``pulling the overcrossings up a little'' gives rise to a link in $\mathbb{R}\mathrm{P}^3$.  With this, we have from   \cite{MR1073213}, that two links in $\mathbb{R}\mathrm{P}^3$ are ambient isotopic if and only if their diagrams are equivalent.

\subsection{Ribbon graphs}\label{s.rg}

\begin{definition}
A {\em ribbon graph} $G =\left(  V(G),E(G)  \right)$ is a (possibly non-orientable) surface with boundary represented as the union of two  sets of  discs, a set $V (G)$ of {\em vertices}, and a set of {\em edges} $E (G)$ such that: 
\begin{enumerate}
\item the vertices and edges intersect in disjoint line segments;
\item each such line segment lies on the boundary of precisely one
vertex and precisely one edge;
\item every edge contains exactly two such line segments.
\end{enumerate}
\end{definition}
An example of a ribbon graph can be found  in Figure~\ref{f1d}, and additional details about them can be found in, for example, \cite{MR3086663,MR1855951}.

Two ribbon graphs are \emph{equivalent} if there is a homeomorphism taking one to the other that sends vertices to vertices, edges to edges, and preserves the cyclic ordering of the edges at each vertex. The homeomorphism should be orientation preserving if the ribbon graphs are orientable. Note that any embedding of a ribbon graph is 3-space is irrelevant.

A ribbon graph is topologically a  surface with boundary and the {\em genus} of a ribbon graph is its genus when it is viewed as a surface. It is {\em orientable} if it is orientable as a surface.  A  ribbon graph is said to be {\em plane} if it is homeomorphic to a  sphere with holes (or equivalently if it is connected and  of genus zero); and  is said to be {\em \RPt} if is homeomorphic to a real projective plane with holes (or equivalently it is connected, non-orientable and of genus one). 

Again since a ribbon graph is a surface with boundary, each ribbon graph $G$ admits a unique (up to homeomorphism) cellular embedding into a closed surface $\Sigma$.  (The cellular condition here means that $\Sigma\backslash G$ is a collection of discs. Using this embedding it is easy to see that ribbon graphs are equivalent to cellularly embedded graphs (in one direction contract the ribbon graph to obtain a graph drawn on the surface, in the other direction take a neighbourhood of the graph in a surface) and so are the main object of topological graph theory. See Figures~\ref{f1d}--\ref{f1f}.


We will make use of the following combinatorial description of ribbon graphs which is due to Chmutov \cite{MR2507944}.
\begin{definition}
An {\em arrow presentation} consists of  a set of closed curves,  each with a collection of disjoint,  labelled arrows, called {\em marking arrows}, lying on them. Each label appears on precisely two arrows.
\end{definition}

A ribbon graph can be obtained from an arrow presentation as follows. View each closed curve as the boundary of a disc (the disc becomes a vertex of the ribbon graph).  Edges are then added to the vertex discs in the following way: take an oriented disc for each label of the marking arrows;  choose two non-intersecting arcs on the boundary of each of the edge discs and direct these according to the orientation; identify these two arcs with two marking arrows, both with the same label, aligning the direction of each arc consistently with the orientation of the marking arrow. This process is illustrated pictorially  in Figure~\ref{f.ap}. 

\begin{figure}[ht]
\centering
\includegraphics[height=10mm]{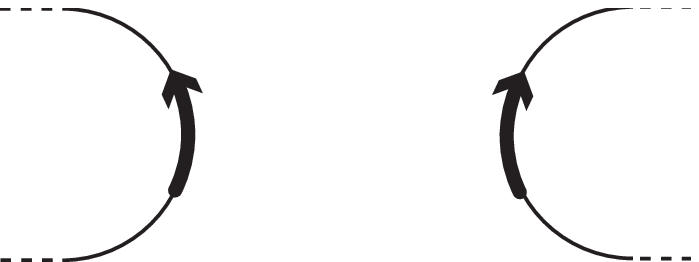} 
\raisebox{4mm}{\hspace{3mm}\includegraphics[width=12mm]{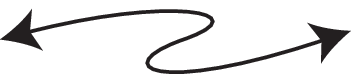}\hspace{3mm}}
\includegraphics[height=10mm]{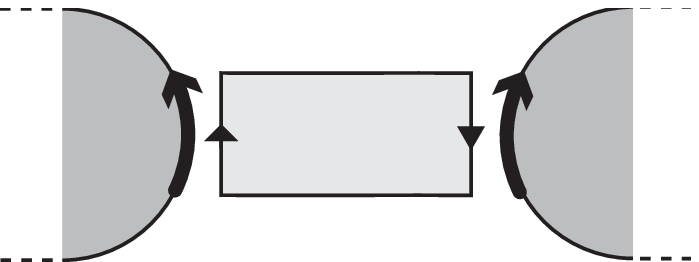} 
\raisebox{4mm}{\hspace{3mm}\includegraphics[width=12mm]{arrow1}\hspace{3mm}}
  \includegraphics[height=10mm]{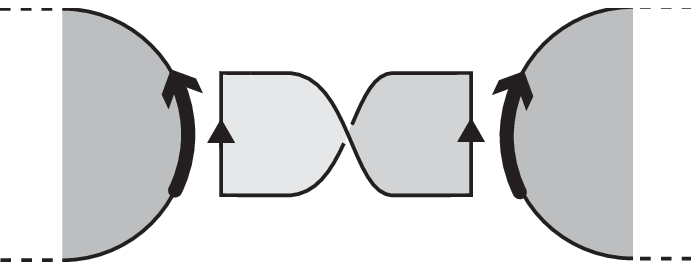} 
  \raisebox{4mm}{\hspace{0mm}\includegraphics[width=12mm]{arrow1}\hspace{3mm}} \includegraphics[height=10mm]{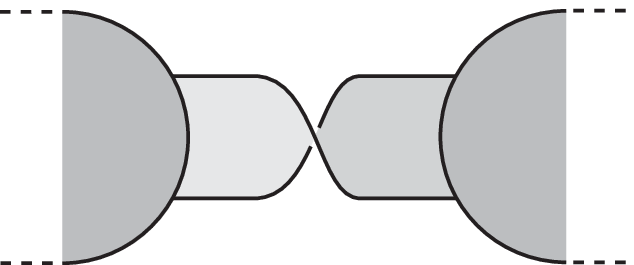}  
  \caption{Moving between arrow presentations and ribbon graphs.}
  \label{f.ap}
\end{figure}

Conversely, to describe a ribbon graph $G$ as an arrow presentation, start by arbitrarily labelling and orienting the boundary of each edge disc of $G$.  On each arc where an edge disc intersects a vertex disc, place an arrow on the vertex disc, labelling the arrow with the label of the edge it meets and directing it consistently with the orientation of the edge disc boundary. The boundaries of the vertex set marked with these labelled arrows give the arrow-marked closed curves of an arrow presentation.  See Figures~\ref{f1c}--\ref{f1d} for an example, and \cite{MR2507944,MR3086663} for further details.

Arrow presentations are {\em equivalent} if they describe equivalent ribbon graphs.

We will need to make use of signed ribbon graphs. A \emph{signed ribbon graph} is a ribbon graph $G$ together with a function from $E(G)$ to $\{+,-\}$. Thus it consists of  a ribbon graph with a sign  associated to each of its edges. Similarly, a \emph{signed arrow presentation} consists of an arrow presentation together with a    function from its set of labels to $\{+,-\}$. Signed ribbon graph and signed arrow presentations are equivalent  in the obvious way.

\section{The ribbon graphs of links in $\mathbb{R}\mathrm{P}^3$}\label{s3}

\subsection{The ribbon graphs of link diagrams}\label{s3a}

We now describe how a set of ribbon graphs can be associated to a link diagram. Let $D$ be a diagram of a link in \RPthreet. 
 Assign a unique label to each crossing of $D$. A {\em marked $A$-splicing} or a   {\em marked $B$-splicing} of a crossing $c$  is the replacement of the crossing with one of the schemes shown in Figure~\ref{f.spl}.

\begin{figure}[ht]
\centering
\subfigure[A Crossing.]{
\labellist
 \small\hair 2pt
\pinlabel {$c$}  at 37 50 
\endlabellist
\quad\quad\includegraphics[scale=0.5]{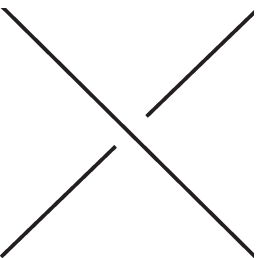} \quad\quad
\label{f.spla}
}
\hspace{5mm}
\subfigure[Its  marked \mbox{A-splicing}. ]{
\labellist
 \small\hair 2pt
\pinlabel {$c-$}  at 30 30  
\pinlabel {$c-$}  at 38 46 
\endlabellist
\quad\quad\includegraphics[scale=0.5]{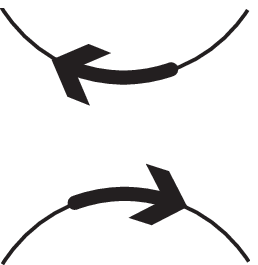} \quad\quad
\label{f.splb}
}
\hspace{5mm}
\subfigure[Its marked \mbox{B-splicing}.]{
\labellist
 \small\hair 2pt
\pinlabel {\rotatebox{-90}{$c+$}}  at 59 31  
 \pinlabel {\rotatebox{-90}{$c+$}}  at 16 38 
\endlabellist
\quad\quad\rotatebox{90}{\includegraphics[scale=0.5]{ch5_28}}\quad\quad
\label{f.splc}
}
\caption{Marked splicings of a link diagram.}
\label{f.spl}
\end{figure}

Notice that we decorate the two arcs in the splicing with signed 
labelled arrows that are chosen to be consistent with an arbitrary orientation of the disc. The labels of the arrows are determined by the label of the crossing, and the signs are 
determined by the choice of splicing.

A {\em state} $\sigma$ of $D$ is the result of  marked $A$- or $B$-splicing  each of its crossings. Observe that a state is  a signed  arrow presentation of a signed ribbon graph. We denote the signed ribbon graph corresponding to the state $\sigma$ of $D$ by $G_{(D,\sigma)}$. These ribbon graphs are the ribbon graphs of a link diagram:
\begin{definition}
Let $D$ be a diagram of a link diagram in \RPthreet. Then the {\em set  of signed ribbon graphs associated with $D$}, denoted $ \G_D$, is defined by 
\[  \G_D =  \{ G_{(D,\sigma)} \mid  \sigma \text{ is a marked state of } D   \}.  \]   
If $G\in \G_D$ then we say that $G$ is a {\em signed ribbon graph of $D$}. We  will also say that $G$ {\em represents} $D$.
\end{definition}
An example of a ribbon graph $G_{(D,\sigma)}$ for a state $\sigma$ of a link diagram $D$ is given in Figures~\ref{f1a}--\ref{f1d}. The construction of $ \G_D$ is a direction extension of the construciton for classical links from \cite{MR2389605,MR925987}.

If $D$ is checkerboard coloured, then we can construct a signed ribbon graph of $D$ by choosing the splicing that follows the black regions at each crossing. The resulting signed ribbon graph is called a \emph{Tait graph} of $D$. If $D$ is checkerboard colourable, then it has exactly two Tait graphs, one corresponding to each of the two checkerboard colourings. 

\begin{proposition}\label{p.tg}
Let $D$ be a checkerboard colourable diagram of a link in \RPthreet. Then its Tait graphs are either plane or \RP ribbon graphs. 
\end{proposition}
\begin{proof} Checkerboard colour $D$ and let $G$ be its Tait graph. 
If $D$ is not null-homologous then all of its regions are discs. Since the marked splicings follow the black regions and the black regions are discs, we can embed $G$ in \RP by taking the black regions bounded by the curves of the splicings as vertices, and embedding the edge disc  between the pairs of labelled arrows in the obvious way. Since $D$ is checkerboard coloured, all regions of the embedded ribbon graph are discs, and no two face regions or vertex regions share a boundary. Thus $G$ is cellularly embedded in the net and is therefore \RPt.  

If $D$ is null-homologous replace the  face of its net that is a M\"obius band with a disc to obtain a diagram on the sphere, and repeat the above  argument with this embedding.  
\end{proof}

We note that it follows from the proof of Proposition~\ref{p.tg} that the Tait graphs defined here coincide with the `usual' Tait graphs obtained by placing vertices in black regions and embedding edges through each crossing.

\begin{remark}\label{rma1}
One of the significant applications of the ribbon graphs of links is that they provide a way to connect graph and knot polynomials.
A seminal result of Thistlethwaite \cite{MR899051},  expresses  the Jones polynomial of an alternating classical link  as an evaluation of the Tutte polynomial of either of  its Tait graphs. There have been several recent extensions of this result that express the Jones polynomial and Kauffaman bracket of virtual and classical links as evaluations of Bollob\'as and Riordans extension of the Tutte polynomial to ribbon graphs, see 
\cite{MR2994589,MR2507944,MR2343139,MR2460170,MR2389605,MR2558980,MR2785760}.
    
The Kauffman bracket and Jones polynomial of links in \RPthree can similarly be expressed in terms of the (multivariate) Bollob\'as-Riordan polynomial of ribbon graphs that represent their diagrams. In fact the statement and proofs of the results for links in  \RPthree follow those for the existing results with almost no change. Accordingly we only remark here that they hold.  Following the notation of the exposition \cite{MR3086663} gives that  for a diagram $D$ of a link in \RPthreet, 
    \[ \langle D\rangle = d^{k(\A)-1} A^{n(\A)-r(\A)}  R(\A; -A^{4}, A^{-2}d, d^{-1},1),  \]
and 
   \[  \langle D\rangle = d^{-1} A^{e_-(G_D)-e_+(G_D)}  Z(G_D; 1, \mathbf{w}, d,1), \quad \text{where } w_e= 
\begin{cases} A^{-2} & \text{if $e$ is negative,}
\\
A^{2} &\text{if $e$ is positive.}
\end{cases}
 \]
In these equations, $\langle D\rangle$ is the Kauffman bracket of \cite{MR1073213},  $d=-A^2-A^{-2}$,  $\A$ is the ribbon graph of $D$ obtained by choosing the A-splicing at each crossing, $R$ is the Bollob\'as-Riordan polynomial of \cite{MR1906909}, and $Z$ is the multivariate Bollob\'as-Riordan polynomial of \cite{MR2368618}. These identities can be obtained by following Section~5.4.2 of  \cite{MR3086663}.

Furthermore a connection between the  Bollob\'as-Riordan polynomial and  the  {\small HOMFLY-PT} of links in \RPthree from \cite{MR2128052}, that is analogous to Jaeger's connection of \cite{MR943099}    between  the Tutte polynomial of a plane graph and the   {\small HOMFLY-PT} polynomial of a classical link (see also \cite{MR2869131,MR2368618,MR955462}), can also be found: 
\[
P(\mathcal{L}(G); x,y ) = \left( \frac{1}{xy} \right)^{v(G)-1}
 \left( \frac{y}{x}\right)^{e(G)}
\left(  x^{2}-1  \right)^{k(G)-1}
  R\left(G; x^{2} , \; \frac{x-x^{-1}}{xy^2},\; \frac{y}{x - x^{-1}}\right).
\]
Again the notation here  is from  \cite{MR3086663}, and the result can be obtained by following Section~5.5.2 of that text.
\end{remark}

\subsection{Relating  link diagrams with the same ribbon graph}
As mentioned in the introduction, two diagrams can give rise the the same set of signed ribbon graphs. That is, it is possible that $D\neq D'$ but   $\G_{D}= \G_{D'}$.  
A fundamental question is then if $D$ and $D'$ are diagrams such that   $\G_{D}= \G_{D'}$,  how are  $D$ and $D'$ related? Here we answer this question in the case when  $D$ and $D'$ are both checkerboard colourable. To describe the result, we need to introduce some notation.

\begin{definition}\label{d.lflip}
Let $D$ and $D'$ be diagrams of links in \RPthreet. We say that $D$ and $D'$ are related by a  {\em summand flip} if  $D'$ can be obtained from $D$ by the following process: orient the disc $D^2$ and choose a disc $\mathfrak{D}$ in $D^2$ whose boundary intersects $D$ transversally in exactly two points $a$ and $b$. Cut out   $\mathfrak{D}$ and glue it back in  such a way that the orientations of $\mathfrak{D}$ and $D^2\bs \mathfrak{D}$ disagree, the points $a$ on the boundaries of $\mathfrak{D}$ and $S^2\bs  \mathfrak{D}$ are identified, and  the points $b$ on the boundaries of $\mathfrak{D}$ and $S^2\bs \mathfrak{D}$ are identified. See Figure~\ref{c5.s6.ss1.f6}. 
We  say that two  link diagrams $D$ and $D'$   are {\em related by summand-flips} if  there is a sequence of summand-flips and R0-moves taking $D$ to $D'$.
\end{definition}

\begin{figure}[ht]
\begin{center}
\begin{tabular}{ccccc}
\labellist
 \small\hair 2pt
\pinlabel {$D$}  at 23 60 
\pinlabel {$D'$}  at 115 60
\endlabellist
\includegraphics[height=2cm]{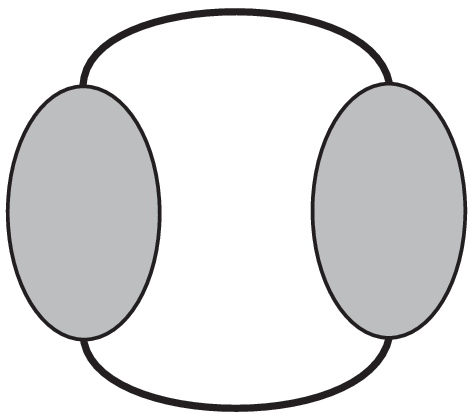} & \raisebox{7mm}{\includegraphics[width=1cm]{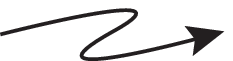} }  & 
\labellist
 \small\hair 2pt
\pinlabel {$D$}  at 23 60 
\pinlabel {$D'$}  at 115 60
\endlabellist
\includegraphics[height=2cm]{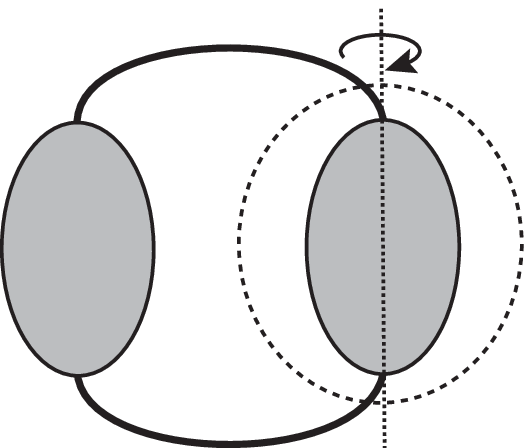}& \raisebox{7mm}{\includegraphics[width=1cm]{ch5_0}}  & 
\labellist
 \small\hair 2pt
\pinlabel {$D$}  at 23 60 
\pinlabel {\reflectbox{$D'$}}  at 115 60
\endlabellist
\includegraphics[height=2cm]{ch5_71} \\
 $D_1$  &&  cut, flip and glue   && $ D_2$
\end{tabular}
\end{center}
\caption{A summand-flip.}
\label{c5.s6.ss1.f6}
\end{figure}

Our first main result is the following.
\begin{theorem}\label{mainthm}
Let $D$ and $D'$ be checkerboard colourable  diagrams of links in  $\mathbb{R}\mathrm{P}^3$. Then $\mathbb{G}_D = \mathbb{G}_{D'}$ if and only if $D$ and $D'$ are related by summand flips. 
\end{theorem}

Before proving Theorem~\ref{mainthm}, we note that the requirement that the link diagrams are checkerboard colourable is essential to our approach, and we pose the following.
\begin{open}
Let $D$ and $D'$ be diagrams of links in  $\mathbb{R}\mathrm{P}^3$ (that are not necessarily checkerboard colourable). Determine necessary and sufficient conditions for   $\mathbb{G}_D$ and $ \mathbb{G}_{D'}$ to be equal.
\end{open}


To prove Theorem~\ref{mainthm} we need to be able to recover link diagrams from ribbon graphs. Given a signed $\rp$ or plane ribbon graph  it is straight-forward to recover a link diagram that it represents.  Let  $G$ be a signed $\rp$ ribbon graph, fill in the holes to obtain a cellular embedding of it in \RPt, as in Section~\ref{s.rg}.  Represent \RP as a disc $D^2$ with antipodal points identified, and lift the embedding of $G$ to a drawing on $D^2$. Finally, draw the  configuration of Figure~\ref{f.rgtol} on each of its edges, and connect the configurations by following the boundaries of the vertices of $G$, to obtain the link diagram. See Figure~\ref{f3} for an example.

\begin{figure}[ht]
   \includegraphics[height=1cm]{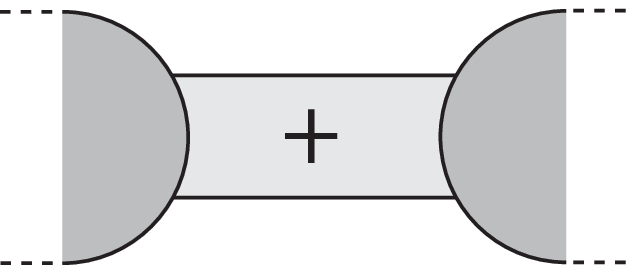}  \raisebox{4mm}{\includegraphics[width=1cm]{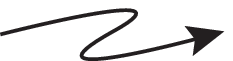}}  \includegraphics[height=1cm]{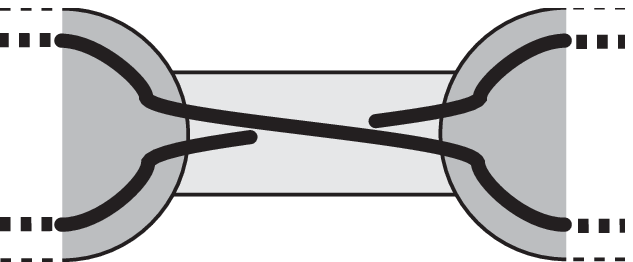}  \hspace{3mm}\raisebox{4mm}{and}  \hspace{3mm}
 \includegraphics[height=1cm]{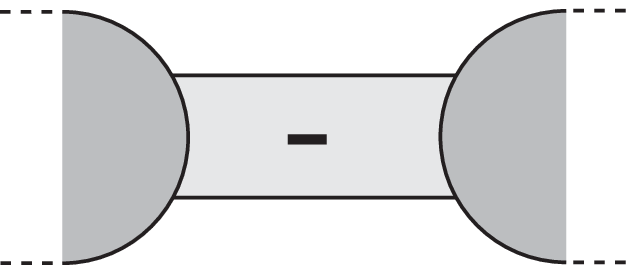}   \raisebox{4mm}{\includegraphics[width=1cm]{arrow}} \includegraphics[height=1cm]{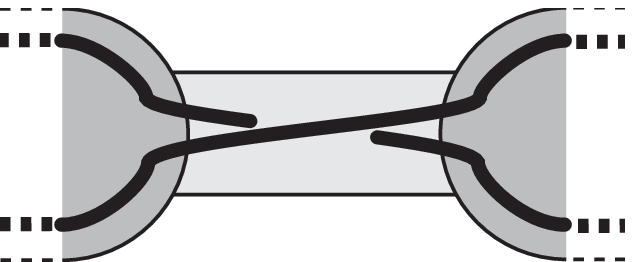}   
\caption{Forming a diagram $D_G$ from a signed ribbon graph $G$.}
\label{f.rgtol}
\end{figure}
If   $G$ is a signed plane ribbon graph, fill in all but one of the the holes to obtain a cellular embedding of it in a disc $D^2$.  Drawing the  configuration of Figure~\ref{f.rgtol} on each of its edges, and connecting the configurations by following the boundaries of the vertices of $G$ gives the required  link diagram. In either case we denote the resulting diagram of a link in \RPthree by $D_G$.

\begin{figure}[ht]
\centering
\subfigure[A ribbon graph $G$.]{
\labellist
 \small\hair 2pt
\pinlabel {1+}  at 150 110
\pinlabel {2-}  at 113 15
\pinlabel {3+}  at 50 129
\endlabellist
\includegraphics[scale=.5]{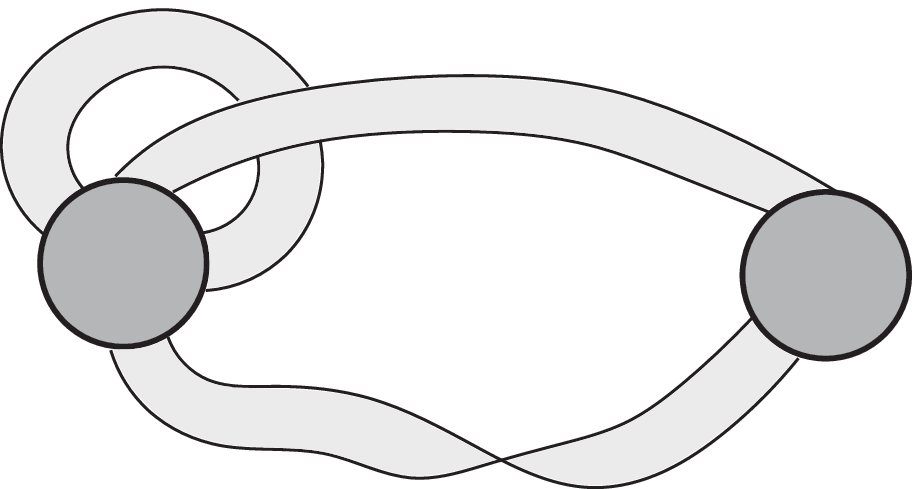}
\label{f3a}
}
\hspace{17mm}
\subfigure[A partial dual $G^{\{3\}}$ of $G$. ]{
\labellist
 \small\hair 2pt
\pinlabel {1+}  at 105 144
\pinlabel {2-}  at 160 100
\pinlabel {3-}  at 68 72
\endlabellist
\includegraphics[scale=.45]{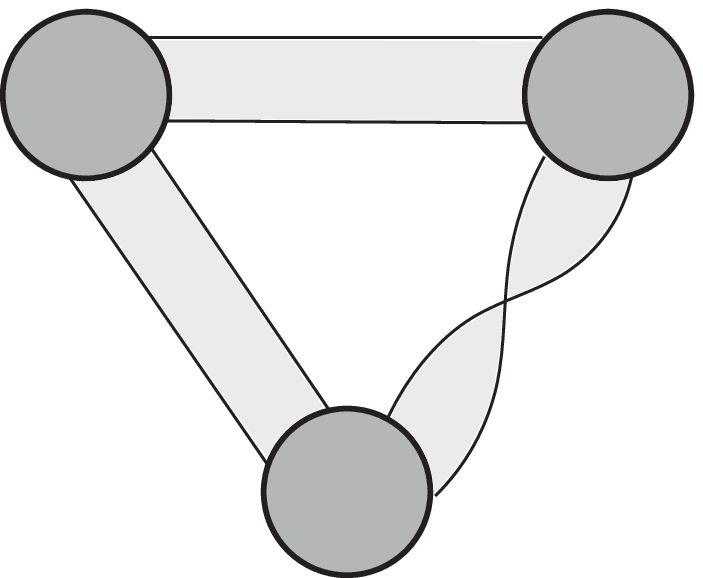}
\label{f3b}
}

\subfigure[Drawing $G^{\{3\}}$ in a disc.]{
\labellist
 \small\hair 2pt
\pinlabel {3-}  at 133 102
\pinlabel {2-}  at 185 102
\pinlabel {1+}  at 78 102
\endlabellist
\includegraphics[scale=.65]{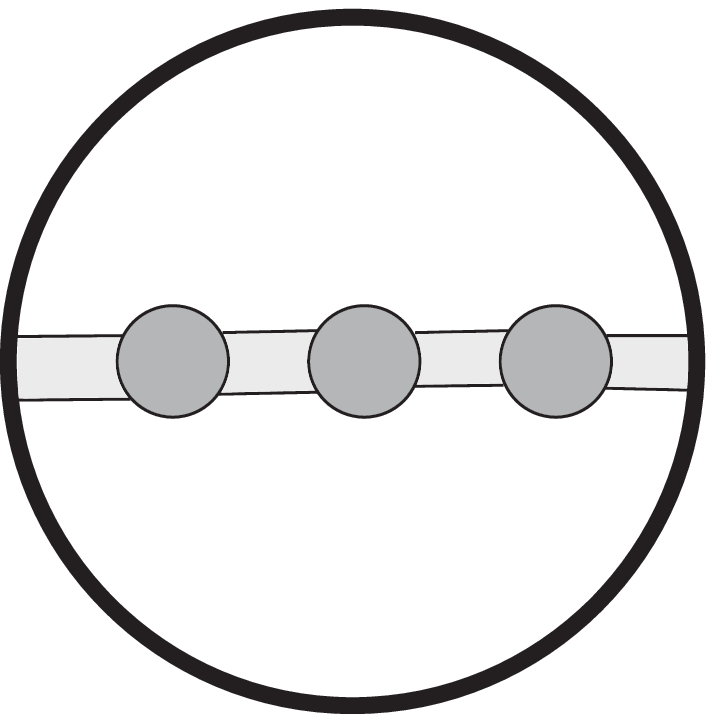}
\label{f3c}
}
\hspace{17mm}
\subfigure[Recovering $D_{G^{\{3\}}}$. ]{
\includegraphics[scale=.65]{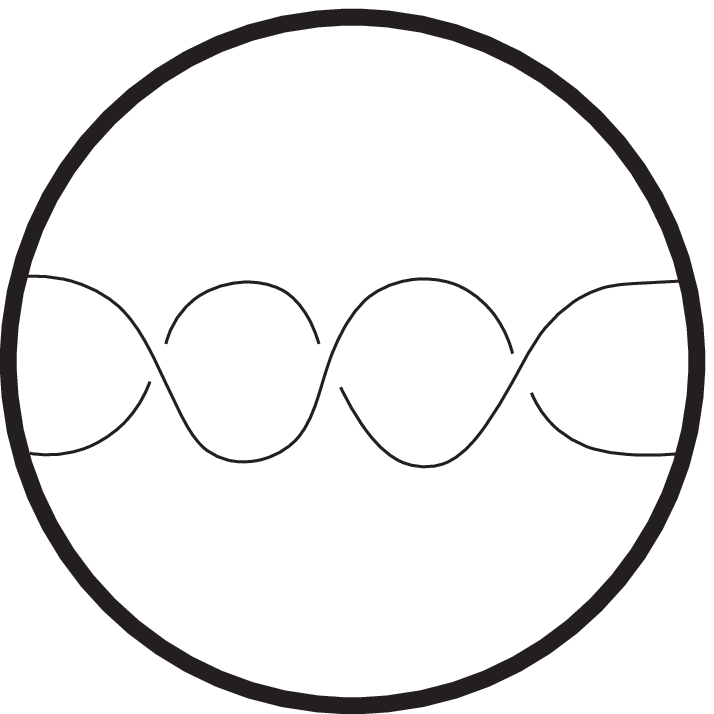}
\label{f3d}
}
\caption{Recovering a link diagram from a ribbon graph.}
\label{f3}
\end{figure}

\begin{proposition}\label{p.1}
 Let $G$ be a signed \RP or plane ribbon graph. Then 
$D_G$ is checkerboard colourable.
\end{proposition}
\begin{proof}
This follows by colouring the regions of $D_G$ that correspond to the vertices of the ribbon graph black.
\end{proof}

To recover a link diagram from a ribbon graph that is not plane or \RP requires  more work, and for our application, Chmutov's concept of a partial dual of a ribbon graph from \cite{MR2507944}.  The idea behind a partial dual is to form the geometric dual of an embedded graph but with respect to only some of its edges. We approach partial duals and geometric duals via arrow presentations as this is particularly convenient for us here. Other descriptions of partial duality can be found in, for example, \cite{MR2507944,MR3086663}. 
\begin{definition}\label{c2.d2} 
Let $G$ be a ribbon graph viewed as an arrow presentation, and let $A\subseteq E(G)$. Then the {\em partial dual} $G^A$ of $G$ with respect to $A$ is the arrow presentation (or ribbon graph) obtained as follows.  For each $e\in A$,  suppose $\alpha$ and $\beta$ are the two arrows labelled $e$ in the arrow presentation of $G$.  Draw a line segment with an arrow on it directed from the head of $\alpha$ to the tail of $\beta$, and a line segment with an arrow on it directed from the head of $\beta$ to the tail of $\alpha$.  Label both of these arrows $e$, and delete $\alpha$ and $\beta$ and the arcs containing them. This process is illustrated locally at a pair of arrows in Figure~\ref{c2.Eop}. The ribbon graph $G^{E(G)}$ is the \emph{geometric dual} of $G$.

If $G$ is a signed ribbon graph then $G^A$ is also a signed ribbon graph with the signs of $G^A$ given by the rule that if an edge $e$ of $G$ has sign $\varepsilon \in \{+,-\}$, then the corresponding edge in $G^A$ has sign  
 $-\varepsilon$ if $e\in A$, and 
  $\varepsilon$ if $e\notin A$. (Thus taking the dual of an edge toggles its sign.)
\end{definition}

\begin{figure}[ht]
\centering
\begin{tabular}{ccc}
\labellist \small\hair 2pt
\pinlabel {$e$} at  52 45   
\pinlabel {$e$}  at   82 30
\endlabellist
 \includegraphics[scale=.5]{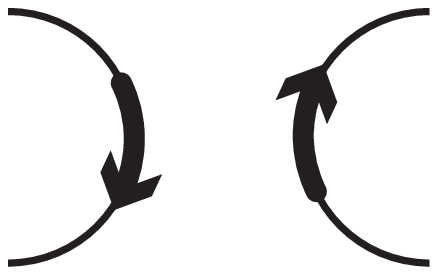} 
 &\quad \raisebox{6mm}{\includegraphics[scale=.4]{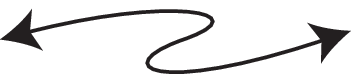}}\quad \quad&
 \labellist \small\hair 2pt
\pinlabel {$e$} at  75 68   
\pinlabel {$e$}  at   55 12
\endlabellist
\includegraphics[scale=.5]{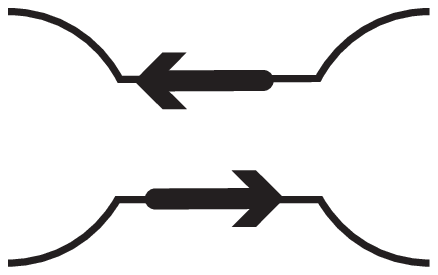}
 \\
$e\in G$ & &$e\in G^{\{e\}}$
\end{tabular}
\caption{Taking the partial dual of an edge in an arrow presentations.}
\label{c2.Eop}
\end{figure}

Figures~\ref{f3a} and \ref{f3b} give an example of a partial dual.
 
We will need the following properties of partial duals from \cite{MR2507944}.
\begin{proposition}\label{c2.p2}
Let $G$ be a (signed) ribbon graph and $A, B\subseteq E(G)$.  Then the following  hold.
\begin{enumerate}
\item \label{c2.p2.1} $G^{\emptyset}=G$.
\item \label{c2.p2.2}  $G^{E(G)}=G^*$, where $G^*$ is the geometric dual of $G$.
\item \label{c2.p2.3} $(G^A)^B=G^{(A\triangle B)}$, where $A\triangle B = (A\cup B)\backslash  (A\cap B)$ is the symmetric difference of $A$ and $B$.
\item \label{c2.p2.4} $G$ is orientable if and only if $G^{A}$ is orientable.
\end{enumerate}
\end{proposition}

We emphasise that the construction of the geometric dual $G^*$ of $G$ agrees with the usual graph theoretic construction of the geometric dual of a cellularly embedded graph in which  a cellularly embedded graph $G^*$ is  obtained from a cellularly embedded graph $G$ by  placing one
vertex in each of its faces, and embedding an edge of $G^*$ between two of these
vertices  whenever the faces of $G$ they lie in are adjacent, and the edges of $G^*$ are embedded so that they cross the corresponding face boundary (or edge of $G$) transversally.

\begin{proposition}\label{p.2}
 Let $G$ be a signed \RP or plane ribbon graph. Then 
$D_G = D_{G^*}$.
\end{proposition}
\begin{proof}
Upon remembering that taking the dual of a signed ribbon graph changes the sign of each edge, the result is readily seen by comparing Figures~\ref{f.rgtol} and~\ref{c2.Eop}.
\end{proof}

\begin{lemma}\label{statespd}
Let $D$ be a diagram of a link in \RPthreet. Then all of the signed ribbon graphs  in $\mathbb{G}_D$ are partial duals of each other.
\end{lemma}

\begin{proof}\label{l.allpd}
Let $G, H \in \mathbb{G}_D$. Then $G = G_{(D, \sigma)}$ and $H = H_{(D, \sigma')}$. It can be seen from Figure~\ref{c2.Eop} that taking partial duals corresponds exactly to choosing another state of $D$ as in Figure~\ref{f.spl}.
\end{proof}


\begin{lemma}\label{l.cbpd}
Let $D$ be a checkerboard colourable diagram of a link in \RPthreet. Then $G$ represents $D$ if and only if $D = D_{G^A}$ where $G^A$ is a signed plane or \RP ribbon graph.
\end{lemma}

\begin{proof}
We begin by assuming that $D = D_{G^A}$ where $G^A$ is a signed plane or \RP ribbon graph. Then $G^A = G_{(D, \sigma)}$ for some   state $\sigma$ of $D$. By Lemma~\ref{statespd} it follows that the partial dual $(G^A)^A=G$ also represents $D$.

Conversely, assume that $G$ represents $D$. Since $D$ is checkerboard colourable, it can  be represented by a Tait graph $T$. Clearly $D=D_T$. Then  from Lemma~\ref{statespd} it follows that $T=G^A$ for some $A\subseteq E(G)$. Since $T$ is a plane or \RP ribbon graph (by Proposition~\ref{p.tg}), $T$ is the ribbon graph required by the lemma.
%
\end{proof}

Lemma~\ref{l.cbpd} provides a way to construct all of the checkerboard colourable link diagrams represented by a given signed ribbon graph: find all of its plane or \RP partial partial duals and construct the links associated with them.  This process is illustrated in Figure~\ref{f3}. The checkerboard colorability requirement here cannot be dropped. For example, if $D$ is the diagram from Figure~\ref{f1a}, then $\mathbb{G}_D$  contains no plane or \RP ribbon graphs. This leads to the following problem.

\begin{open}
Let $G$ be a signed ribbon graph. Find an efficient way to construct all of the diagrams of links in \RPthree that have $G$ as a representative.
\end{open}
We continue with some  corollaries of Lemma~\ref{l.cbpd}.

\begin{corollary}\label{c.nh}
Let $D$ and $D'$ be checkerboard colourable  diagrams of links in \RPthree such that  $\mathbb{G}_D = \mathbb{G}_{D'}$, then $D$ is null-homologous if and only if  $D'$ is.
\end{corollary}
\begin{proof}
$D$ is null-homologous if and only if  it has a plane Tait graph. The result then follows since partial duality preserves orientability.
\end{proof}

\begin{corollary}\label{c.dd}
Let $D$ and $D'$ be checkerboard colourable  diagrams of links in \RPthree such that  $\mathbb{G}_D = \mathbb{G}_{D'}$. Then there exists a plane or, respectively,  \RP  ribbon graph $G$, and $A \subseteq E(G)$ such that $G^A$ is plane  or, respectively, \RP  and such that $D = D_G$ and $D' = D_{G^A}$.
\end{corollary}

\begin{proof}
We have that $D$ and $D'$ give rise to the same set of ribbon graphs. Since $D$ is checkerboard colourable, it gives rise to a plane or \RP ribbon graph $G$ (namely one of its Tait graphs, by Proposition~\ref{p.tg}). Moreover, since $D'$ is also checkerboard colourable, it also gives rise to a plane or \RP ribbon graph $H$. We also have that $H \in \mathbb{G}_D$, so $H = G^A$ for some $A \subseteq E(G)$ by Lemma \ref{statespd}.
\end{proof}

Corollary~\ref{c.dd} is of key importance here: it tells us that if two   checkerboard colourable  diagrams of links in \RPthreet, $D$ and $D'$, are represented by the same ribbon graphs, then they are both diagrams associated with partially dual plane or  \RP ribbon graphs  $G$ and $G'$. Thus if we understand how   $G$ and $G'$ are related to each other, we can deduce how $D$ and $D'$ are related to each other. This is our strategy for proving Theorem~\ref{mainthm}. 

In  \cite{MR2928906} and  \cite{MR2994404},  rough structure theorems for the partial duals of plane ribbon graphs and \RP ribbon graphs were given. These papers also contained  local moves that allows us to move between all partially dual plane or \RP ribbon graphs. To describe this move, we need a little additional terminology. 

Let $G$ be a ribbon graph, $v\in V(G)$, and $P$ and $Q$ be non-trivial ribbon subgraphs of $G$. Then $G$ is said to be the {\em join} of $P$ and $Q$, written $P\vee Q$, if $G=P\cup Q$ and $P\cap Q=\{v\}$ and if there exists an arc on $v$ with the property that all edges of $P$ meet it there, and none of the edges of $Q$ do. 
See the left-hand side of Figure~\ref{c5.s6.ss1.f4} which illustrates a ribbon graph of the form $P\vee Q$.
We do not require the ribbon graphs $G$, $P$ or $Q$ to be connected.   Note that since genus is additive under joins, if $G$ is plane then both  $P$ and $Q$ are plane, and if $G$ is \RP then exactly one of  $P$ or $Q$ is \RP and the other is plane.

 Let $G=P\vee Q$ be a ribbon graph. We  say that the ribbon graph   $G^{E(Q) } =  P \vee Q^{ E(Q)}=P \vee Q^*$ is obtained from $G$ by a  {\em dual-of-a-join-summand move}. We say that two ribbon graphs are related by {\em dualling  join-summands} if there is a sequence of  dual-of-a-join-summand moves taking one to the other, or if they are geometric duals.  See Figure~\ref{c5.s6.ss1.f4}.

\begin{figure}
\centering
\labellist
 \small\hair 2pt
 \pinlabel  {$\overbrace{\hspace{44pt}}^{P}$}  [l] at 0 105 
  \pinlabel  {$\underbrace{\hspace{50pt}}_{Q}$}  [l] at 50 0 
\endlabellist
\includegraphics[scale=.6]{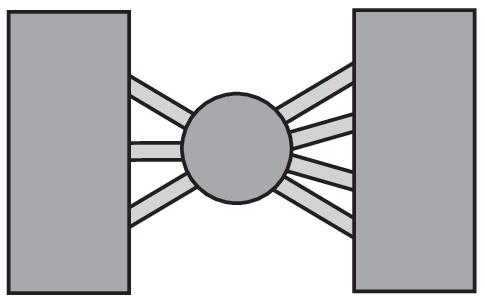}
\quad \raisebox{7mm}{\includegraphics[scale=.6]{ch5_0}} \quad
\labellist
 \small\hair 2pt
  \pinlabel  {$\overbrace{\hspace{44pt}}^{P}$}  [l] at 0 105 
  \pinlabel  {$\underbrace{\hspace{44pt}}_{ Q^*}$}  [l] at 45 0 
\endlabellist
\includegraphics[scale=.6]{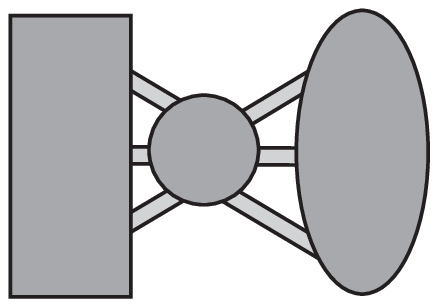}
\caption{The dual of  a join-summand move.}
\label{c5.s6.ss1.f4}
\end{figure}

The following result is an amalgamation of Theorem~7.3 of \cite{MR2928906} and Theorem~5.8 of \cite{MR2994404}.
\begin{theorem}\label{t.sim}
Let $G$ and $H$ be connected plane or \RP ribbon graphs. Then $G$ and $H$ are partial duals if and only if they are related by dualling  join-summands.
\end{theorem}

Theorem~\ref{t.sim} allows us to prove the following key result.
\begin{lemma}\label{lmain}
If  two \RP ribbon graphs $G$ and $G'$  are related by dualling join-summands, then the link diagrams $D_G$ and $D_{G'}$ they represent  are related by summand-flips.
\end{lemma}

\begin{proof} 
It suffices to show that if $G$ and $G'$  are related by a single dual-of-a-join-summand move then $D_G$ and $D_{G'}$ are related by a summand-flip. Suppose that $G = A\vee B$, that $A\cap B = \{v\}$, and that $G'=A^*\vee B$ or  $G'=A\vee B^*$.   Since we know that genus is additive under joins, we have that one of $A$ or $B$ is \RP and the other is plane. Without loss of generality, suppose that $A$ is the \RP summand. 
 
First suppose that $G'=A^*\vee B$. We start by determining how the cellular embeddings of  $G$ and $G'$ are related. From this we will deduce how the corresponding link diagrams are related. Start by taking the cellular embedding of $G$ in \RPt. This is illustrated in  Figure~\ref{lemma1r}. For each edge of $B$ that meets $v$, place a labelled arrow on the intersection of the edge with $v$. We can then `detach' $B$ from $G$, as indicated in Figure~\ref{wtf}, so that $G$ is recovered from $A$ and $B$ by identifying the corresponding arrows in $A$ and in $B$ with its copy of $v$ removed.  After detaching $B$ we obtain a cellular embedding of $A$ in \RPt. 
From this form the cellular embedding of $A^*$ by interchanging the vertices and faces. (In detail, $A^*  \subset \rp$ is obtained from $A  \subset \rp$ by reassigning the face (respectively, vertex) discs of $A  \subset \rp$ as vertex (respectively, face) discs of $A^*  \subset \rp$. Edge discs are unchanged.) This is indicated in Figure~\ref{lemma3r}. Finally, obtain an embedding of $G'=A^*\vee B $ by reattaching $B$ according to the labelled arrows, as is indicated in arrows as in Figure~\ref{lemma4r}, and notice that $B$ has been `flipped over'.  Finally consider the diagrams $D_G$ and $D_{G'}$ drawn using these embeddings. Since $A$ and $A^*$ have the same edges and vertex/face boundaries, and by Proposition~\ref{p.2}, $D_A=D_{A^*}$, and we see that $D_G$ and $D_{G'}$ are related by a summand-flip, as in Figures~\ref{lemma5r} and~\ref{lemma6r}.

Next suppose that $G'=A\vee B^*$. Then, using Proposition~\ref{c2.p2} and that duality preserves joins, we have $G'= (A \vee B^*) = (A \vee B^*)^{**} = (A^* \vee B^{**})^*= (A^* \vee B)^* $. Then since $D_{(A^* \vee B)^* }=D_{(A^* \vee B) }$, by Proposition~\ref{p.2}, this case reduces to the first, completing the proof.
\end{proof}

\begin{figure}
\centering
\subfigure[$G  \subset \rp$]{
\labellist
 \small\hair 2pt
\pinlabel {$A$}  at 110 50
\pinlabel {$B$}  at 102 140
\pinlabel {$v$}  at 150 101
\endlabellist
\includegraphics[scale=.5]{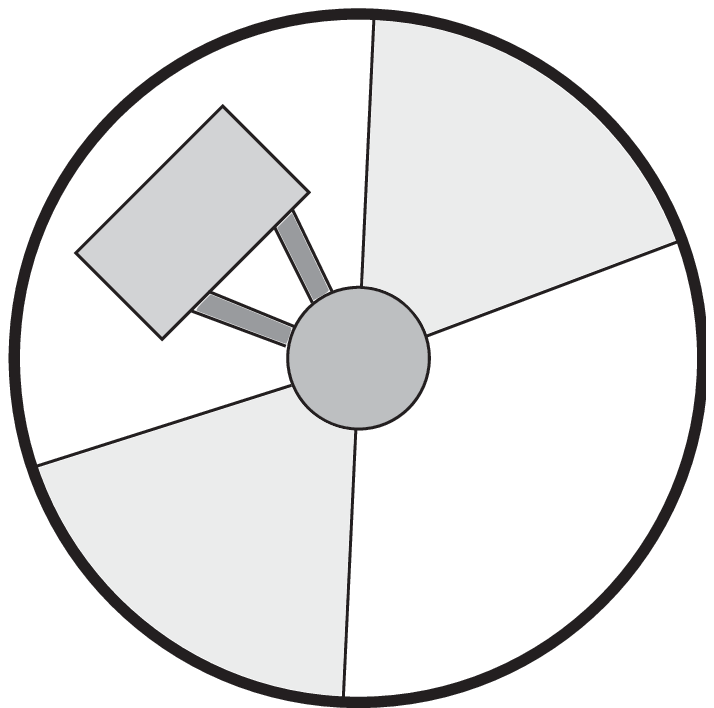}
\label{lemma1r}
}
\hspace{17mm}
\subfigure[$A  \subset \rp$]{
\labellist
 \small\hair 2pt
\pinlabel {$A$}  at 120 50
\pinlabel {$B$}  at 33 191
\pinlabel {$v$}  at 160 101
\endlabellist
\includegraphics[scale=.5]{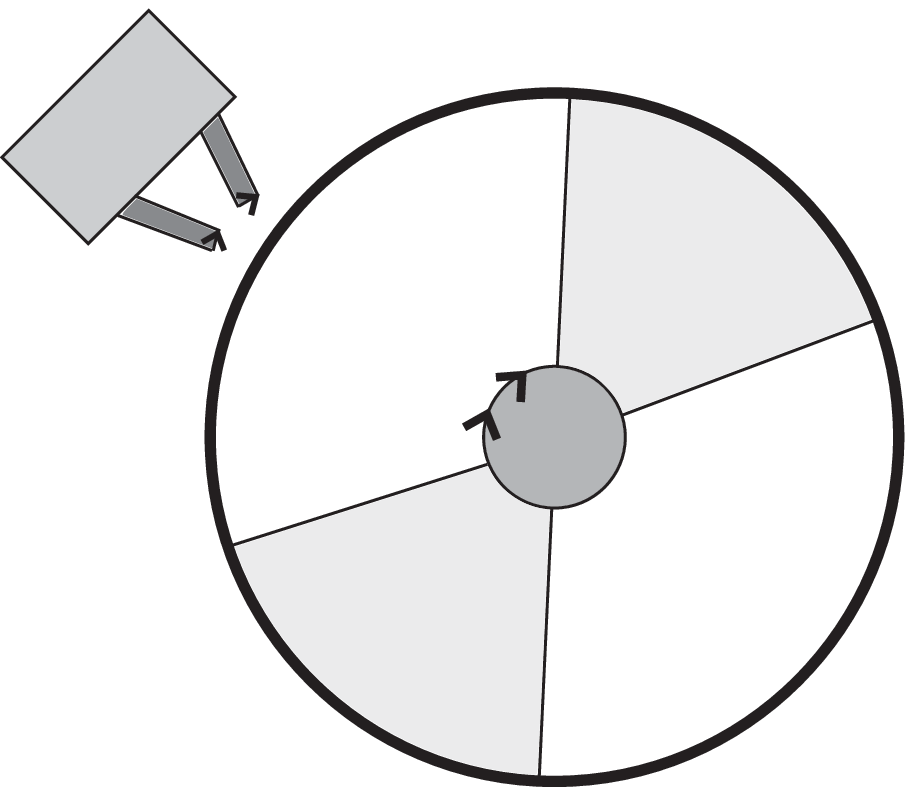}
\label{wtf}
}

\subfigure[$A^*  \subset \rp$]{
\labellist
\small\hair 2pt
\pinlabel {$A^*$}  at 100 140
\pinlabel {$B$}  at 33 188
\endlabellist
\includegraphics[scale=.5]{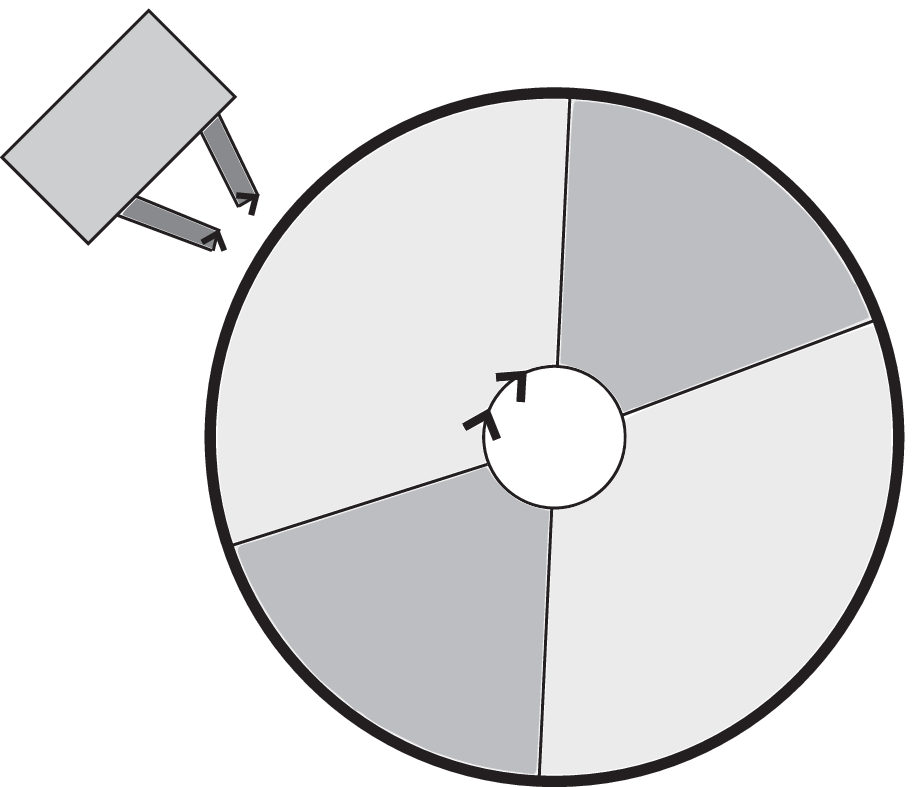}
\label{lemma3r}
}
\hspace{17mm}
\subfigure[$G'$]{
\labellist
\small\hair 2pt
\pinlabel {$A^*$} at 90 143
\pinlabel \scalebox{1}[-1]{$B$} at 153 105
\endlabellist
\includegraphics[width=8cm, height=8cm,scale=0.5]{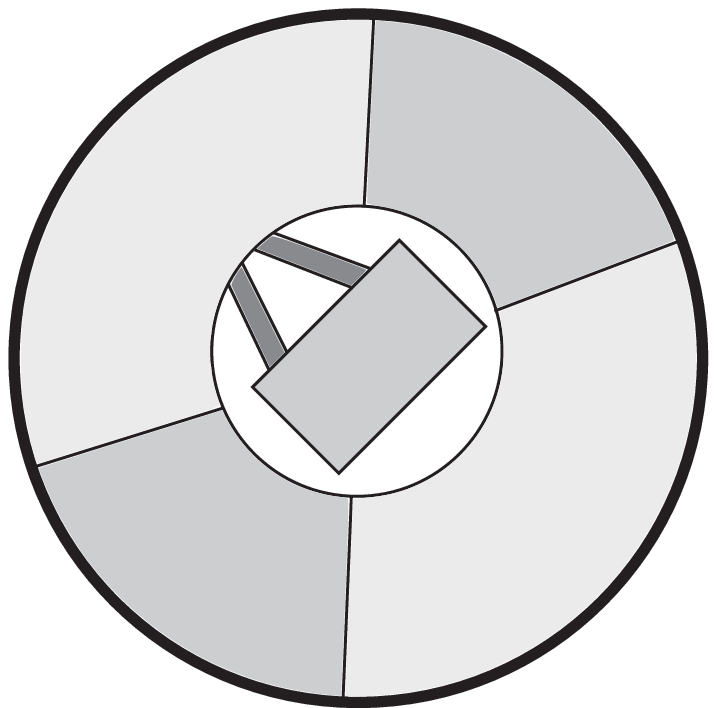}
\label{lemma4r}
}

\subfigure[$D_G$]{
\labellist
\small\hair 2pt
\pinlabel {$\alpha$} at 107 50
\pinlabel {$\beta$} at 103 143
\endlabellist
\includegraphics[scale=0.5]{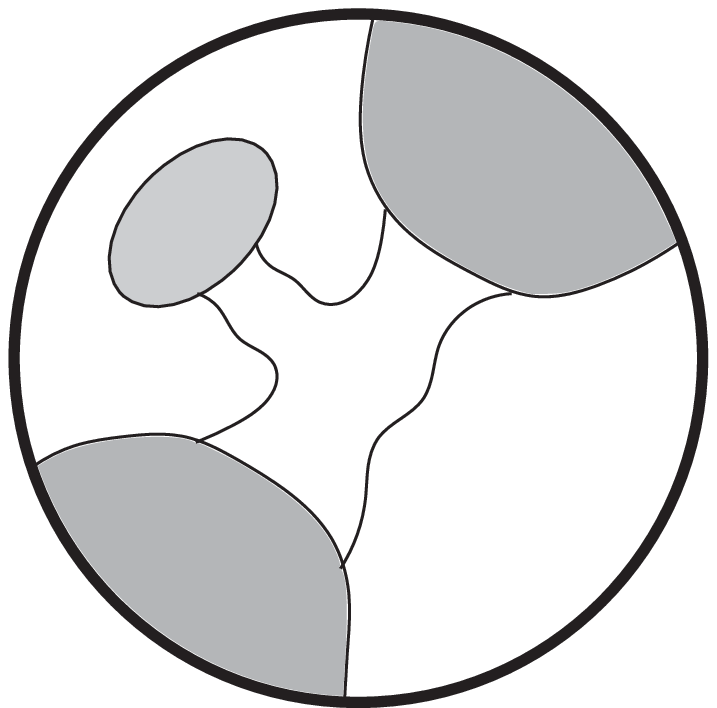}
\label{lemma5r}
}
\hspace{17mm}
\subfigure[$D_{G'}$]{
\labellist
\small \hair 2pt
\pinlabel {$\alpha$} at 107 50
\pinlabel \scalebox{1}[-1]{$\beta$} at 143 107
\endlabellist
\includegraphics[scale=0.5]{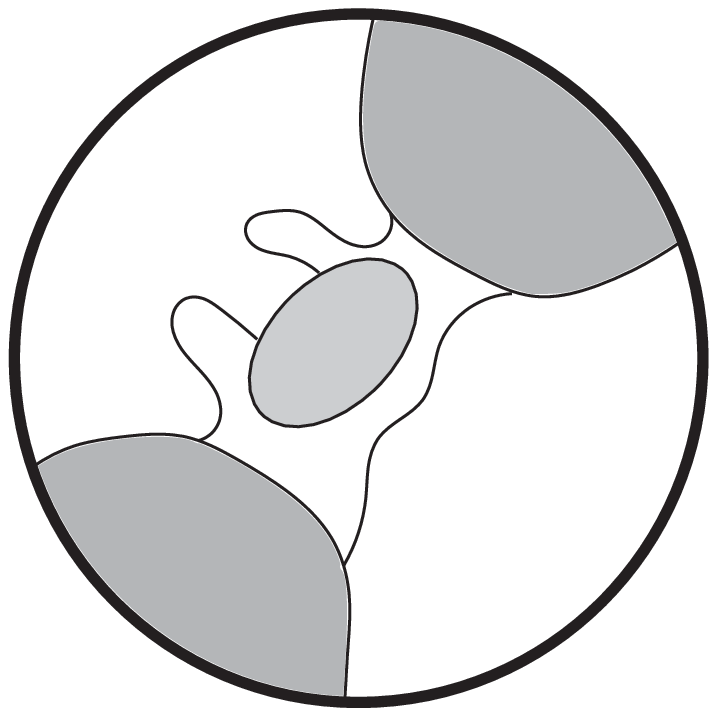}
\label{lemma6r}
}
\caption{A figure used in the proof of Lemma~\ref{lmain}.}
\label{L3r}
\end{figure}

\begin{proof}[Proof of Theorem~\ref{mainthm}]
It is readily seen that if $D$ and $D'$ are related by summand-flips then $\G_D=\G_{D'}$.
 
For the converse, assume that $D$ and $D'$ are checkerboard colourable link diagrams on \RP such that $\mathbb{G}_D = \mathbb{G}_{D'}$. If $D$ and $D'$ are not null-homologous then,  by Corollary~\ref{c.dd}, for some $G$ we have  $D = D_G$ and $D' = D_{G^A}$ where $G$ and $G^A$ are both \RP. We  know by  Theorem~\ref{t.sim} that $G$ and $G^A$ are related by dualling join-summands. Thus either  $G^A = G^*$, in which case the result follows from Proposition~\ref{p.2}, or  $G^A$ is obtained from $G$ by a sequence of dual-of-a-join-summand moves, in which case the result follows from Lemma~\ref{lmain}.
\end{proof}

\section{One vertex ribbon graphs}\label{1v}
We let $\A_D$ denote the \emph{all-A ribbon graph} of $D$ which is the ribbon graph obtained from $D$ by choosing the marked A-splicing at each crossing. The all-A ribbon graph is of particular interest since all of the signs are the same, and so a link diagram can be represented by an unsigned ribbon graph (see also Remark~\ref{rma1}). It was shown in \cite{MR3148509} that every classical link (i.e., in $S^3$)   can be represented as a ribbon graph with exactly one vertex. Furthermore, the authors of that paper gave a set of moves, analogous to the Reidemeister moves, that provide a way to move between all of the diagrams of a classical link that have one-vertex all-A ribbon graphs. In this section we extend their result to links in \RPthreet.

\begin{lemma}
Every link in \RPthree has a diagram $D$ for which $\A_D$ has exactly one vertex.
\end{lemma}

\begin{proof}
Let $D$ be a diagram of  a  link  in \RPthreet.  Let $\sigma_A$ denote the \emph{all-A state} of $D$ obtained by  choosing the marked A-splicing at each crossing. If $\sigma_A$ has exactly one component then $\A_D$ has exactly one vertex. Otherwise, consider the all-A state $\bar{\sigma}_A$ of the net $\mathcal{N}_D$ of $D$.  There must be two closed curves of $\bar{\sigma}_A$ that can be joined by an embedded arc $\bar{\alpha}$  in $\mathbb{R}\mathrm{P}^2 \setminus \bar{\sigma}_A$. Performing an RII-move (possibly with some RIV-moves)  along the image of this arc in $D$ gives a new diagram $D'$. Then $\A_{D'}$  has one vertex less that $\A_D$. Repeat this process until only one curve remains.
\end{proof}

The \emph{M-moves} for diagrams of links in $\mathbb{R}\mathrm{P}^3$ consist of isotopy of the disc that preserves the antipodal pairing, together with the  moves shown in Figure~\ref{M-moves} that change the diagram locally as shown (the diagrams are identical outside of the shown region). For the $M_0$-move, we require the diagram to be connected in a specific way, as indicated by the labels.

\begin{figure}[ht]
	\centering
	\labellist
	\small\hair 2pt	
	\pinlabel {$M_0$} at 235 520
	\pinlabel{$a$} at 0 520
	\pinlabel{$b$} at 50 540
	\pinlabel{$b$} at 97 540
	\pinlabel{$a$} at 145 520
	\pinlabel{$c$} at 42 455
	\pinlabel{$c$} at 102 455
	
	\pinlabel{$a$} at 350 520
	\pinlabel{$b$} at 404 540
	\pinlabel{$b$} at 445 548
	\pinlabel{$a$} at 529 520
	\pinlabel{$c$} at 387 455
	\pinlabel{$c$} at 515 455
	
	\pinlabel{$M_1$} at 240 365
	
	\pinlabel{$M_2$} at 240 223
	
	\pinlabel{$M_3$} at 240 106
	
	\pinlabel{$M_4$} at 864 610
	
	\pinlabel{$M_5$} at 875 380
	
	\pinlabel{$M_6$} at 880 120
	
	\endlabellist
	\includegraphics[scale=0.4]{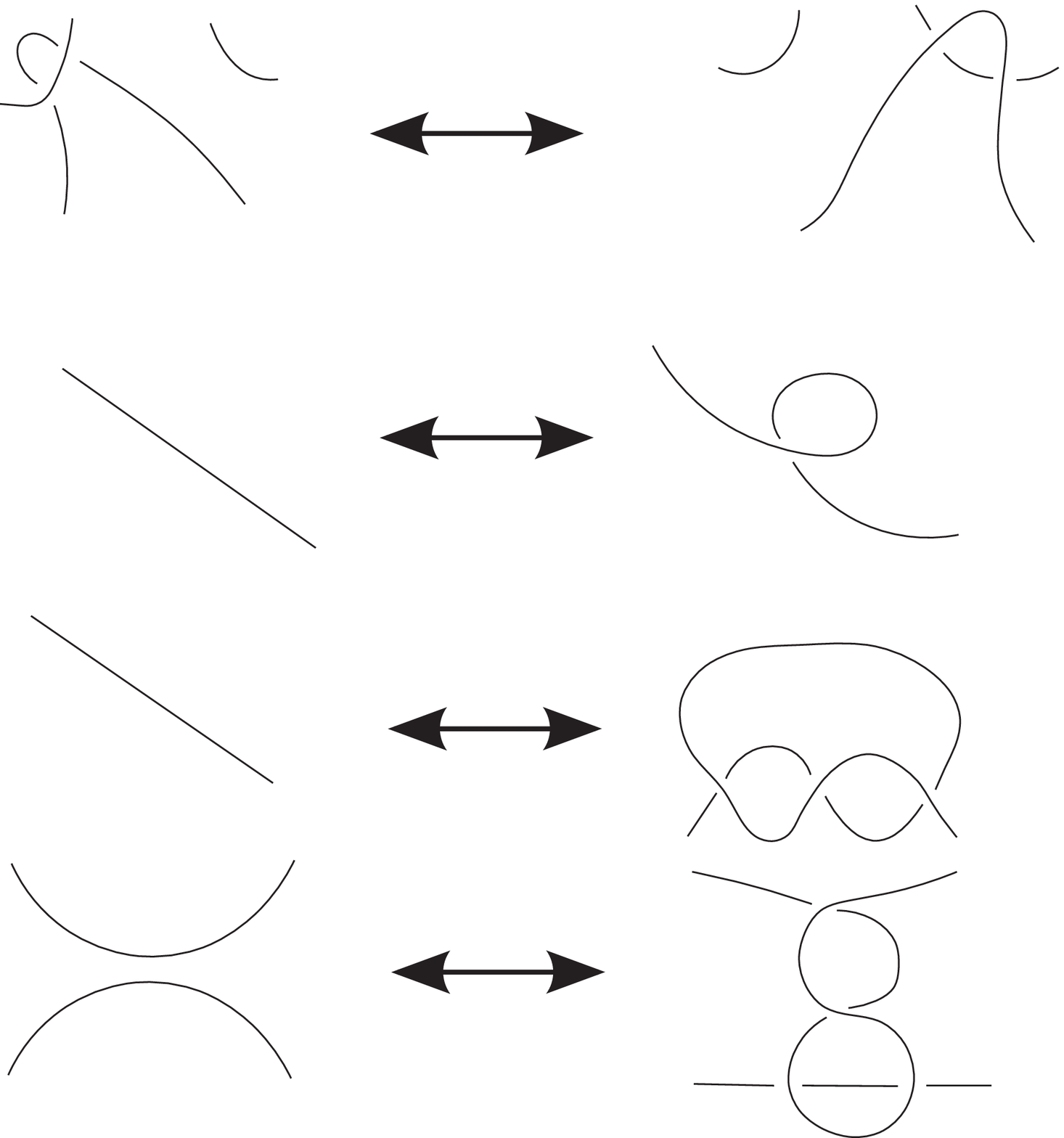}
	\hspace{10mm}
	\includegraphics[scale=0.37]{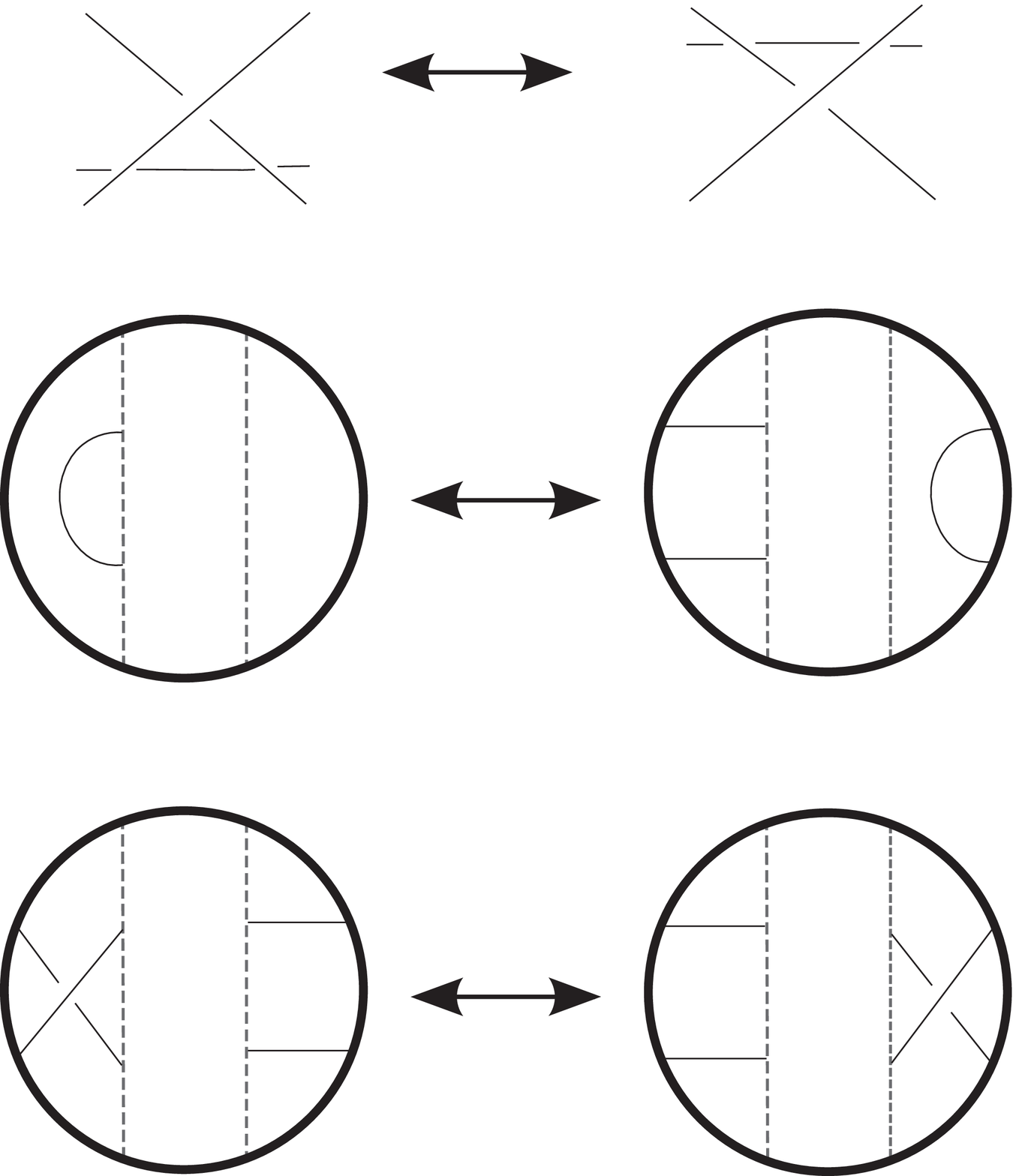}
	\caption{The M-moves.}
	\label{M-moves}
\end{figure}

\begin{lemma}
Let $D$ be a diagram of  a given  link  in \RPthreet. Then the M-moves do not change the number of vertices in $\A_D$.
\end{lemma}

\begin{proof}
For moves $M_{0}$--$M_4$ we refer the reader to \cite{MR3148509}. It is easy to see that the $M_5$ move does not affect the number of components of the all-A state $\sigma_A$ of $D$, since it does not affect the number of, or type of, crossings.
It is also easy to see that $M_6$ does not change the number of vertices of the all-A ribbon graph.
\end{proof}

Let $\mathcal{D}$ denote the set of all diagrams of links in \RPthreet,   $\tilde{\mathcal{D}}$ denote $\mathcal{D}$  modulo the Reidemeister moves,  $\mathcal{D}_1 \subset \mathcal{D}$ denote the subset of diagrams such that their all-A ribbon graphs have exactly one vertex, and $\tilde{\mathcal{D}}_1$ denote  $\mathcal{D}_1 $ modulo the M-moves. Now consider the two natural projections $\phi: \mathcal{D} \to \tilde{\mathcal{D}}$ and $\phi_1 : \mathcal{D}_1 \to \tilde{\mathcal{D}_1}$.

\begin{theorem}
Given $D, D' \in \mathcal{D}_1$, then $\phi(D) = \phi(D')$ if and only if $\phi_1(D) = \phi_1(D')$.
\end{theorem}
\begin{proof}
First assume that $\phi_1(D) = \phi_1(D')$. Then the link diagrams are related by M-moves. It is easy to see that the link diagrams are then related by Reidemeister moves, so we have that $\phi(D) = \phi(D')$.

Conversely, suppose that $\phi(D) = \phi(D')$. Hence the diagrams are related by Reidemeister moves. We need to show that each Reidemeister move can be describes as a sequence of M-moves. For RI--RIII, we refer the reader to \cite{MR3148509}.  RIV and RIV are exactly $M_5$ and $M_6$ moves, so we have that all the Reidemeister moves can be described as a sequence of $M$-moves. Hence $\phi_1(D) = \phi_1(D')$, as required.
\end{proof}

\section{Virtual link diagrams with same the signed ribbon graphs.}\label{s.virt}
A {\em virtual link diagram}  consists of $n$ closed piecewise-linear plane curves in which there are finitely many multiple points and such that at each multiple point exactly two arcs meet and they meet transversally. Moreover, each double point is assigned either a {\em classical crossing} structure or is marked as a {\em virtual crossing}.  See the left-hand side of Figure~\ref{c5.s2.ss3.f2} where the virtual crossings are marked by circles. A virtual link is {\em oriented} if each of its plane curves is. Further details on virtual knots can be found in, for example, the surveys \cite{MR1721925,MR1865707,MR2994594,MR2255973,MR2068425}.

\begin{figure}[ht]
\centering
\labellist
 \small\hair 2pt
\pinlabel {$1$}  at 91 95 
\pinlabel {$2$}  at 52 95 
\pinlabel {$3$}  at 72 6 
\pinlabel {$1$}  at 225 14 
\pinlabel {$2$}  at 206 102 
\pinlabel {$3$}  at 258 58 
\pinlabel {$+$}  at 200 77 
\pinlabel {$+$}  at 232 43 
\pinlabel {$-$}  at 237 66 
\endlabellist
\includegraphics[scale=0.7]{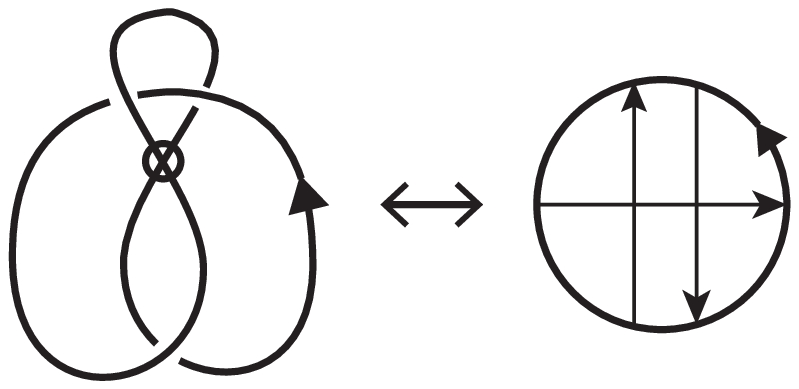}
\caption{A virtual link (on the left) and its Gauss diagram (on the right). The crossings and chords are numbered for clarity.}
\label{c5.s2.ss3.f2}
\end{figure}

Virtual links are considered up to the {\em generalised Reidemeister moves}. These consist of orientation preserving homeomorphisms of the plane (which we include in any subset of the moves), the classical Reidemester moves of Figure~\ref{f2a}, and the virtual Reidemeister moves of Figure~\ref{vrm}.
 Two virtual link diagrams are {\em equivalent} if there is a sequence of generalised Reidemeister moves taking one diagram to the other.

\begin{figure}[ht]
\centering
\labellist
 \small\hair 2pt
\pinlabel {vRIII}  at 108 50
\pinlabel {vRI}  at 108 158
\pinlabel {vRII}  at 398 158
\pinlabel {RIII}  at 398 50
\pinlabel {mixed}  at 398 65
\endlabellist
\includegraphics[scale=.55]{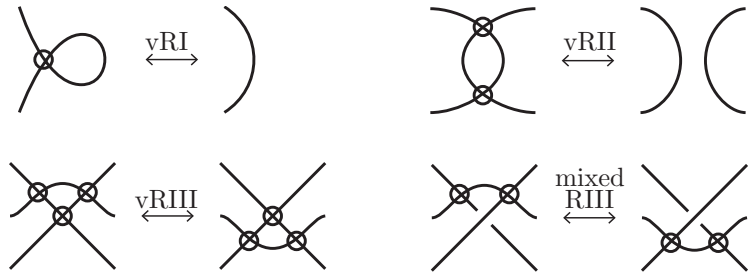}
\caption{The virtual Reidemeister moves.}
\label{vrm}
\end{figure}

Virtual knots are the knotted objects that can be represented by Gauss diagrams.
Here  a {\em Gauss diagram} consists of a set of oriented circles together with a set of oriented signed chords whose end points lie on the circles (see the right-hand side of Figure~\ref{c5.s2.ss3.f2}). 
A Gauss diagram is obtained from an oriented $n$ component virtual link diagram $D$ as follows. Start by numbering each classical crossing.  For each component, choose a base point and travel round the component from the base point  following the orientation and reading off the numbers of the classical crossings as they are met. Whenever a crossing is met as an over-crossing, label the corresponding number with the letter O.  Place each number, in the order met, on an oriented circle corresponding to the component. Connect the points on the circles that have the same number by a chord that is directed away from the O-labelled number. Finally, label each chord with the {\em oriented sign} of the corresponding crossing, shown in Figure~\ref{c5.s2.ss3.f1},  and delete the numbers. The resulting Gauss diagram describes $D$. See Figure~\ref{c5.s2.ss3.f2} for an example.

\begin{figure}[ht]
\centering
\subfigure[A positive crossing.]{
\hspace{10mm}\includegraphics[scale=.5]{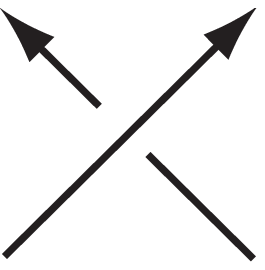}\hspace{10mm}
\label{c5.s2.ss3.f1a}
}
\hspace{10mm}
\subfigure[A negative crossing. ]{
\hspace{10mm}\includegraphics[scale=.5]{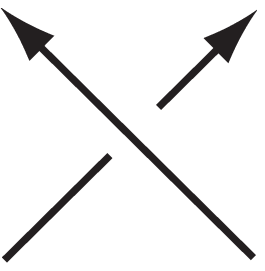}\hspace{10mm}
\label{c5.s2.ss3.f1b}
}

\caption{The oriented signs of a link diagram.}
\label{c5.s2.ss3.f1}
\end{figure}

Conversely, an oriented virtual link diagram can be obtained from a Gauss diagram by immersing the circles in the plane so that the ends of chords are identified (there is no unique way to do this), and using the direction and signs to obtain a crossing structure. In general, immersing the circles will create double points that do not arise from chords. Mark these as  virtual crossings.  

The following theorem of Goussarov, Polyak and Viro from \cite{MR1763963} provides an important and fundamental relation between Gauss diagrams and virtual links.
\begin{theorem}\label{c5.s1.ss3.t2}
Let $L$ and $L'$ be two virtual link diagrams that are described by the same Gauss diagram. Then $L$ and $L'$ are equivalent. 
Moreover, $L$ and $L'$ are related by the virtual Reidemeister moves.
\end{theorem}

In \cite{MR2460170}, Chmutov and Voltz  observed that the construction of a ribbon graph from a link diagram can be extended to include virtual links. That is, if $D$ is a virtual link diagram and $\sigma$ is a state of $D$, then $G_{(D,\sigma)}$, and the set $\mathbb{G}_D$ can be associated with $D$ just as in Section~\ref{s3a} (virtual crossings are not smoothed, and the curves of the arrow presentation follow the component of the virtual link through the virtual crossings).

In Theorem~\ref{mainthm} we determined how diagrams of links in \RPthree that are represented by the same set of ribbon graphs are related. We will now consider the corresponding problem for virtual links. We start by determining which ribbon graphs represent virtual link diagrams.

If $G$ is a signed ribbon graph, then we can recover a virtual link diagram $D$ with $G=G_D$ as follows: delete the interiors of the vertices of $G$ (so that we obtain a set of ribbons that are attached to circles). Immerse the resulting object in the plane in such a way that the ribbons are embedded. (Note that as the circles are immersed, they may cross each other and themselves.) Replace each embedded ribbon  with a classical crossing with the crossing structure determined by the sign as in Figure~\ref{f.rgtol}. Make all of the intersection points of the immersed circles into virtual crossings. See Figure~\ref{c5.s6.ss2.f1}.
The resulting virtual link diagram $D$ has the desired property that $G=G_D$ (as $G$ can be obtained for $D$ by reversing the above construction). Moreover, every virtual link diagram that is represented by $G$ can be obtained in this way. This follows since if $G=G_D$, then we can go through the above process drawing the  circles and crossings in such a way that they follow $D$. 

Thus we have that every signed ribbon graph is the signed ribbon graph of some virtual link diagram.

\begin{figure}[ht]
\centering
\labellist
 \small\hair 2pt
\pinlabel {$2+$}  at 168 136
\pinlabel {\rotatebox{-90}{$1-$}}  at 13 120 
\pinlabel {$3-$}  at 134 17 
\endlabellist
\includegraphics[scale=.45]{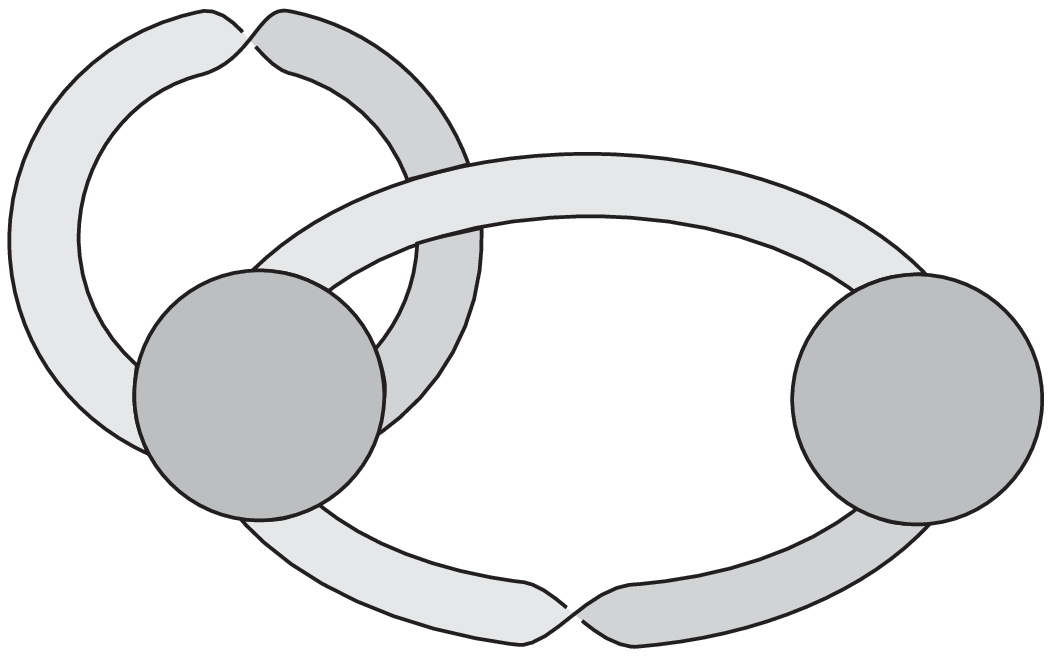}
\quad \raisebox{11mm}{\includegraphics[scale=.6]{ch5_0}} \quad
\labellist
 \small\hair 2pt
\pinlabel {$2+$}  at 168 136
\pinlabel {\rotatebox{-90}{$1-$}}  at 13 120 
\pinlabel {$3-$}  at 134 17 
\endlabellist
\includegraphics[scale=.45]{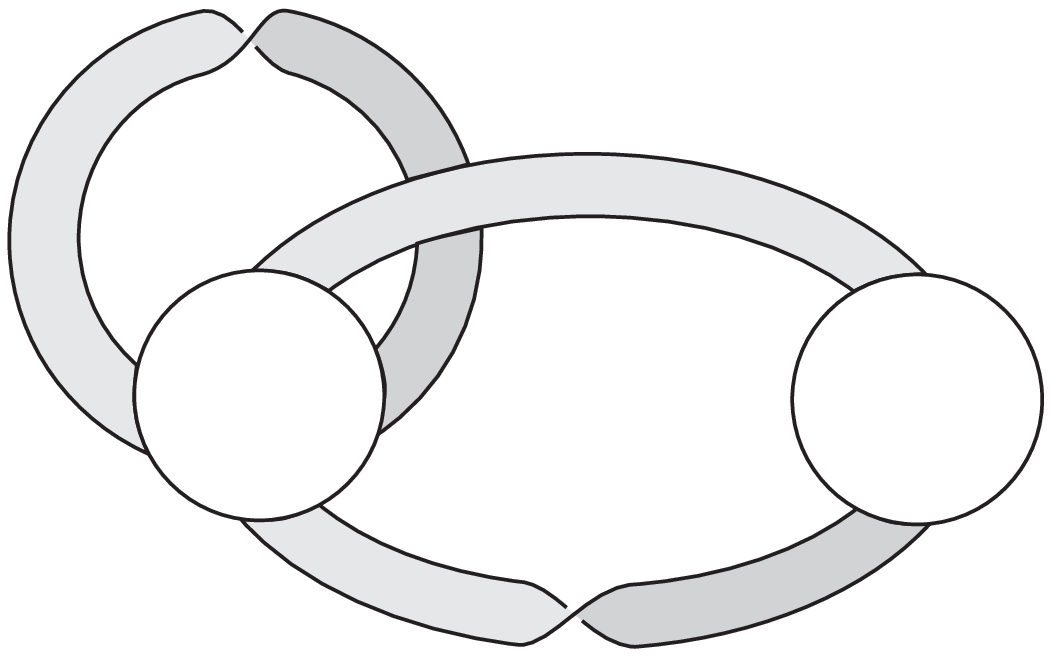}

\quad \raisebox{11mm}{\includegraphics[scale=.6]{ch5_0}} \quad
\labellist
 \small\hair 2pt
\pinlabel {$2+$}  at 104 91
\pinlabel {$1-$}  at  104 165
\pinlabel {$3-$}  at  104 18
\endlabellist
\includegraphics[scale=.45]{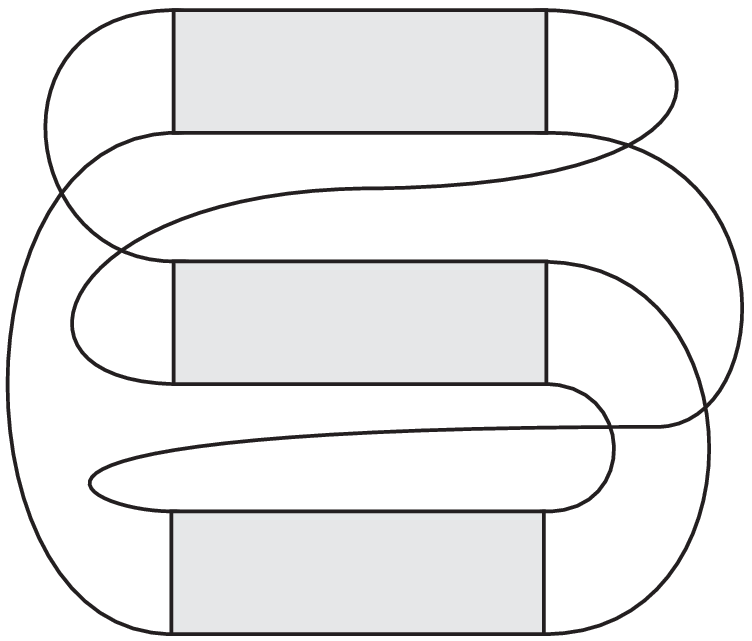}
 \raisebox{11mm}{\includegraphics[scale=.6]{ch5_0}} \quad
\includegraphics[scale=.45]{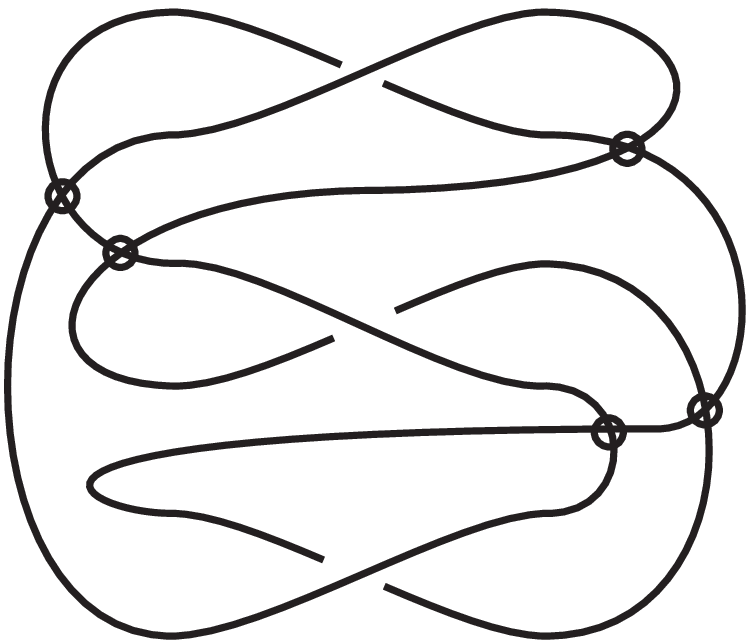}
\caption{Recovering a virtual link diagram from a signed ribbon graph.}
\label{c5.s6.ss2.f1}
\end{figure}

We now determine how virtual link diagrams that are represented by the same ribbon graphs are related. For this we need the concept of  virtualisation.
The {\em virtualisation} of a crossing of a virtual link diagram is the flanking of the crossing with  virtual crossings as indicated in  Figure~\ref{c5.s6.ss2.f2}. The crossing in the figure can also be of the opposite type. 
\begin{figure}[ht]
\centering
\includegraphics[scale=.4]{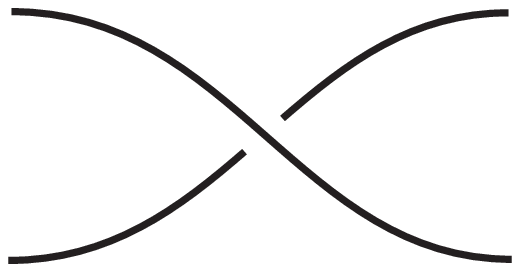}
\quad \raisebox{3mm}{\includegraphics[scale=.8]{ch5_0}} \quad
\includegraphics[scale=.4]{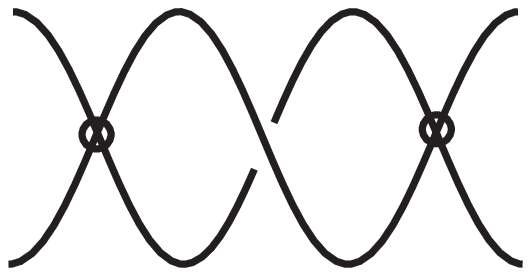}
\caption{Virtualising a crossing.}
\label{c5.s6.ss2.f2}
\end{figure}

\begin{theorem}\label{t.vmain}
Let $D$ and $D'$ be two virtual link diagrams. Then $D$ and $D'$   are presented by the same set of  signed ribbon graphs if and only if they are related by virtualisation and the virtual Reidemeister moves. 
\end{theorem}
\begin{proof} 
Let $G$ be a signed ribbon graph. Label and arbitrarily orient each edge of $G$. As described above, every virtual link diagram represented by $G$ can be obtained by (1) deleting  the interiors of the vertices of $G$, (2) embedding the edges of $G$ in the plane, (3) immersing the arcs connecting the edges (note that arcs in an immersion may cross each other), and (4) adding the crossing structure as described above. 

Suppose $D$ and $D'$ are two virtual link diagrams obtained from $G$ by this procedure.
If the edges of $G$ are oriented, in step (2) each embedding of an edge either agrees or disagrees with the orientation of the plane. If in step (2) of the constructions of $D$ and $D'$ the corresponding edges either both agree or both disagree with the orientation of the plane, it is easily seen that for some orientation of their components (in each diagram choose orientations that agree at each pair of  crossings that correspond the the same edge of the ribbon graph),  $D$ and $D'$ must then be described by the same Gauss diagram.   In this case, by  Theorem~\ref{c5.s1.ss3.t2}, they are related by the purely virtual moves and the semivirtual move. 

Now suppose that in step (2) of the construction of $D$ and $D'$, there is an edge $e$ of $G$ such that 
the orientations of the two plane embeddings  disagree  with each other, and otherwise the embeddings of the edges and immersions of the arcs in step (3) are identical. Then, by Figure~\ref{c5.s6.ss2.f3}, the resulting virtual link diagrams are related by virtualisation.

It then follows that if $G$ is a signed ribbon graph then the link diagrams it represents are related by virtualisation and the  virtual moves. The converse of the theorem is easily seen to hold.
\end{proof}

\begin{figure}[ht]
\centering
\labellist
 \small\hair 2pt
 \pinlabel {$+$}  at  100 37
\pinlabel {$a$}  at  3 64
\pinlabel {$b$}  at   3 10
\pinlabel {$d$}  at  194 64
\pinlabel {$c$}  at  194 10
\endlabellist
\includegraphics[height=17mm]{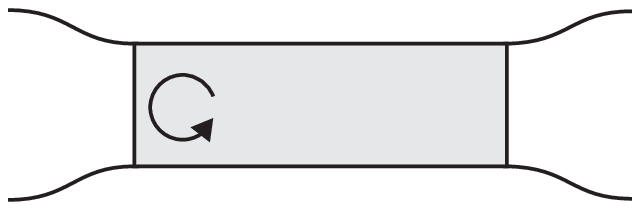}
\quad \raisebox{7mm}{\includegraphics[scale=.6]{ch5_0}} \quad
\labellist
 \small\hair 2pt
\pinlabel {$a$} [r] at  0 78
\pinlabel {$b$} [r] at   0 6
\pinlabel {$d$} [l] at  150 78
\pinlabel {$c$} [l] at  150 6
\endlabellist
\includegraphics[height=15mm]{ch5_84}

\vspace{5mm}

\labellist
 \small\hair 2pt
 \pinlabel {$+$}  at  100 37
\pinlabel {$a$}  at  3 64
\pinlabel {$b$}  at   3 10
\pinlabel {$d$}  at  194 64
\pinlabel {$c$}  at  194 10
\endlabellist
\includegraphics[height=17mm]{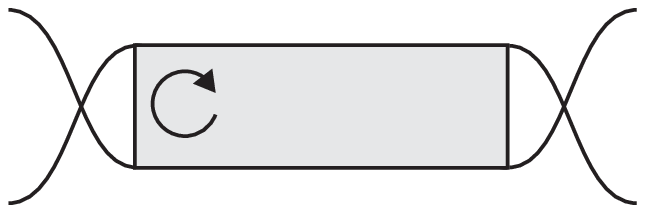}
\quad \raisebox{7mm}{\includegraphics[scale=.6]{ch5_0}} \quad
\labellist
 \small\hair 2pt
\pinlabel {$a$} [r] at  0 78
\pinlabel {$b$} [r] at   0 6
\pinlabel {$d$} [l] at  150 78
\pinlabel {$c$} [l] at  150 6
\endlabellist
\includegraphics[height=15mm]{ch5_83}

\caption{Forming virtual link diagrams from a signed ribbon graph.}
\label{c5.s6.ss2.f3}
\end{figure}

\bibliographystyle{amsplain}  
\bibliography{references}

\end{document}